\documentclass[12pt]{amsart}

\author{Simeon Ball}\thanks{
The author acknowledges the support of the project MTM2014-54745-P of the Spanish {\em Ministerio de Econom\'ia y Competitividad.}}

 \usepackage{amsmath}
\usepackage{amssymb}
\usepackage{amsthm}
\usepackage{graphicx}
\usepackage{amscd}

\newtheorem{theorem}{Theorem}
\newtheorem{lemma}[theorem]{Lemma}

\title{On sets defining few ordinary planes}

 \usepackage{amsmath}
\usepackage{amssymb}
\usepackage{amsthm}
\usepackage{graphicx}
\usepackage{amscd}

\begin{document}

\baselineskip=17pt

\date{21 June 2017}

\maketitle


\begin{abstract}
Let $S$ be a set of $n$ points in real three-dimensional space, no three collinear and not all co-planar. We prove that if the number of planes incident with exactly three points of $S$ is less than $Kn^2$ for some $K=o(n^{\frac{1}{7}})$ then, for $n$ sufficiently large, all but at most $O(K)$ points of $S$ are contained in the intersection of two quadrics. Furthermore, we prove that there is a constant $c$ such that if the number of planes incident with exactly three points of $S$ is less than $\frac{1}{2}n^2-cn$ then, for $n$ sufficiently large, $S$ is either a prism, an anti-prism, a prism with a point removed or an anti-prism with a point removed. As a corollary to the main result, we deduce the following theorem. Let $S$ be a set of $n$ points in the real plane. If the number of circles incident with exactly three points of $S$ is less than $Kn^2$ for some $K=o(n^{\frac{1}{7}})$ then, for $n$ sufficiently large, all but at most $O(K)$ points of $S$ are contained in a curve of degree at most four.
\end{abstract}

\section{Introduction}

This work is inspired by the article of Green and Tao \cite{GT2013} in which they classify sets $S$ of $n$ points in the real plane with the property that the number of lines incident with exactly two points of $S$ is less than $Kn$,  where $K<c(\log \log n)^c$ for some constant $c$ and $n$ is sufficiently large. The lines incident with exactly two points of $S$ are called {\em ordinary lines}. 

The natural generalisation in three-dimensional space is to attempt to classify sets $S$ of $n$ points where the number of planes incident with exactly three points of $S$ is less than $Kn^2$ for some $K$. The planes incident with exactly three points of $S$ are called {\em ordinary planes}. However, it is easy to construct sets for which there are no ordinary planes, if one is allowed to put more than three points on a line. Indeed, as observed by Motzkin \cite{Motzkin1951}, if $S$ has $r \geqslant 4$ points on one line and $|S|-r$ points on another line, then there will be no planes incident with exactly three points of $S$. 

In this article we will consider a set $S$ of $n$ points, no three collinear and not all co-planar, with the property that the number of planes incident with exactly three points of $S$, the number of ordinary planes spanned by $S$, is less than $Kn^2$, for some $K=K(n)$. Since we will be considering the dual space, we suppose $S$ to be a set of points in the three-dimensional real projective space, which we will denote by $\mathrm{PG}_3({\mathbb R})$. This is no restriction since, given that $S$ is finite, we can always find a plane which is incident with no point of $S$. By changing the basis, we can then assume that this plane is the plane at infinity, so that $S$ will be contained in the affine space. We will use the notation $(x_1,x_2,x_3,x_4)$ to denote a projective point, so the one-dimensional subspace spanned by the vector with these coordinates. 

The problem of minimising the number of hyperplanes incident with exactly $d$ points of a  set $S$ of $n$ points in $d$-space, where $S$ has the property that any subset of $d$  points spans a hyperplane, is a natural generalisation of Sylvester's problem, see \cite{BM2016}. In \cite{Syl1893}, Sylvester asks if it is possible to find a finite set $S$ of points in the real plane, not all collinear, with the property that no line is incident with precisely two points of $S$. In \cite{Gallai1944}, Gallai proves that it is not possible and many articles (for example \cite{KM1958} and more recently, \cite{CS1993} and \cite{BM1990}) have been published concerning the problem of determining the minimum number of ordinary lines spanned by a set of $n$ points. The results of Green and Tao from \cite{GT2013} resolve this problem for sufficiently large $n$. More importantly, they prove asymptotic structural theorems for finite sets of points in the plane spanning few ordinary lines. Here, we are interested in obtaining asymptotic results for three-dimensional space. The main results of this article are the following theorems.

\begin{theorem} \label{inter} 
Let $S$ be a set of $n$ points in $\mathrm{PG}_3({\mathbb R})$, no three collinear and not all co-planar, spanning less than $Kn^2$ ordinary planes, for some $K=o(n^{\frac{1}{7}})$. Then one of the following holds:
\begin{enumerate}
\item[(i)]
There are two distinct quadrics such that all but at most $O(K)$ points of $S$ are contained in the intersection of the quadrics. Furthermore, all but at most $O(K)$ points of $S$ are incident with at least $\frac{3}{2}n-O(K)$ ordinary planes.
\item[(ii)]
There are two planar conic sections of a quadric which contain at least $\frac{1}{2}n-O(K)$ points of $S$.
\item[(iii)]
All but at most $2K$ points of $S$ are contained in a plane.
\end{enumerate}

\end{theorem}

In the following theorem, by strengthening the hypothesis, we classify the extremal examples. In the case when $n$ is even and sufficiently large, it was already proven in \cite{Monserrat2015}, using \cite[Theorem 3.8]{GT2013} of Green and Tao, that the extremal examples are a prism or an anti-prism.

\begin{theorem} \label{detailed} 
Let $S$ be a set of $n$ points in $\mathrm{PG}_3({\mathbb R})$, no three collinear and not all co-planar. There is a constant $c$ such that if $S$ spans less than $\frac{1}{2}n^2-cn$ ordinary planes then, for $n$ sufficiently large, $S$ is either a prism, an anti-prism, a prism with a point removed or an anti-prism with a point removed.
\end{theorem}

In the following section we define precisely what we mean by a prism and an anti-prism.

\section{The prism and the anti-prism} \label{prism}

Let $\pi_1$ and $\pi_2$ be distinct planes of $\mathrm{PG}_3({\mathbb R})$. Let $R_1$ be a set of points projectively equivalent to the vertices of a regular $m$-gon in $\pi_1$ and let $R_2$ be a set of points projectively equivalent to the vertices of a regular $m$-gon in $\pi_2$. If there is a point $p_{\infty}$ with the property that for every point $p \in R_1$, the line joining $p$ and $p_{\infty}$ intesects $\pi_2$ in a point of $R_2$ then we say that $R_1 \cup R_2$ is a {\em prism}.

For example, taking $\pi_1$ to be the plane $X_3=0$ and $\pi_2$ to be the plane $X_4=0$, we let 
$$
p_i=(\cos (2\pi i/m), \sin (2 \pi i/m), 0 , 1 )
$$ 
and define
$$
R_1=\{ p_i \ | \ i=0,\ldots,m-1 \},
$$
let 
$$
q_k=(\cos (2\pi k/m)-a, \sin (2 \pi k/m)-b,  1, 0 )
$$
for some $a,b \in {\mathbb R}$ and define
$$
R_2=\{ q_k  \ | \ k=0,\ldots,m-1 \}.
$$
Then $p_{\infty}$ is the projective point $(a,b,-1,1)$.

It follows from the sum-angle formulas for sine and cosine that the line joining the points $p_i$ and $p_j$ meets the line $\ell_{\infty}= \pi_1 \cap \pi_2$ in the point
$$
s_{\theta}=(\sin (\pi \theta/m), \cos (\pi \theta/m),0,0),
$$
where $\theta=-(i+j)$ mod $m$. Likewise, the line joining the points $q_k$ and $q_{\ell}$ meets the line $\ell_{\infty}= \pi_1 \cap \pi_2$ in the point $s_{\theta}$, where $\theta=-(k+\ell)$ mod $m$.

Let $S$ be the union of $R_1$ and $R_2$. The three points $p_i$, $p_j$ and $q_k$ and the three points $q_i$, $q_j$ and $p_k$ span ordinary planes if and only if there is no $\ell \neq k$ such that $i+j=k+\ell$ mod $m$, which is if and only if $i+j=2k$ mod $m$. 

If $m$ is odd then $i+j=2k$ mod $m$ has $\frac{1}{2}m(m-1)$ solutions where $i \neq j$. Therefore, if $n=2$ mod $4$ then a prism with $n$ points spans $2\frac{1}{2}(\frac{1}{2}n)(\frac{1}{2}n-1)=\frac{1}{4}n^2-\frac{1}{2}n$ ordinary planes. 

If $m$ is even then $i+j=2k$ mod $m$ has $\frac{1}{2}m^2-m$ solutions where $i \neq j$. Therefore, if $n=0$ mod $4$ then a prism with $n$ points spans $\frac{1}{4}n^2-n$ ordinary planes.

\begin{figure}[h]
\centering
\includegraphics[width=4.6 in]{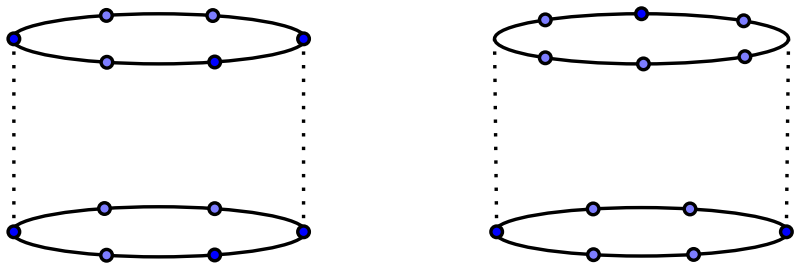}
\caption{The prism and the anti-prism with 12 points.}
\label{prisms}
\end{figure}

Let 
$$
r_k=(\cos (2\pi (k+ \tfrac{1}{2})/m)-a, \sin (2 \pi (k+ \tfrac{1}{2})/m)-b, 1,0 )
$$
for some $a,b \in {\mathbb R}$ and define
$$
R_3=\{ r_k  \ | \ k=0,\ldots,m-1 \}.
$$
Then the line joining the points $r_k$ and $r_{\ell}$ meets the line $\ell_{\infty}= \pi_1 \cap \pi_2$ in the point $s_{\theta}$, where $\theta=-(k+\ell+1)$ mod $m$.

Let $S$ be the union of $R_1$ and $R_3$. The three points $p_i$, $p_j$ and $r_k$ span an ordinary plane if and only if there is no $\ell \neq k$ such that $i+j=k+\ell+1$ mod $m$, which is if and only if $i+j=2k+1$ mod $m$. The three points $r_i$, $r_j$ and $p_k$ span an ordinary plane if and only if there is no $\ell \neq k$ such that $i+j+1=k+\ell$ mod $m$, which is if and only if $i+j=2k-1$ mod $m$. 

If $m$ is odd then $i+j-1=2k$ mod $m$ and $i+j+1=2k$ both have $\frac{1}{2}m(m-1)$ solutions where $i \neq j$. Therefore, if $n=2$ mod $4$ then an anti-prism with $n$ points spans $\frac{1}{4}n^2-\frac{1}{2}n$ ordinary planes. 

If $m$ is even then $i+j-1=2k$ mod $m$ and $i+j+1=2k$ both have $\frac{1}{2}m^2$ solutions where $i \neq j$. Therefore, if $n=0$ mod $4$ then an anti-prism with $n$ points spans $\frac{1}{4}n^2$ ordinary planes.

\section{Extending Melchior's use of Euler's formula}

In this section we use Melchior's approach to Sylvester's problem from  \cite{Melchior1940}, which uses Euler's formula for planar graphs, generalising it to three dimensions. 

Let $S$ be a set of $n$ points in $\mathrm{PG}_3({\mathbb R})$, no three collinear and not all co-planar.

In the dual space the set $S$ dualises to a set $S^{*}$ of $n$ planes. 
Any $3$ planes of $S^{*}$ intersect in a point, since any $3$ points of $S$ span a plane. Let $V$ denote this set of points. Let $L$ denote the set of lines cut out by the intersection of two planes of $S^{*}$, i.e. the dual of the lines spanned by two points of $S$.

Denote by $\Gamma=\Gamma(S)$ the graph embedded in the three dimensional projective space which has vertices $V$ and edges $E$ which are the line segments induced from the lines of $L$, cut out by the vertices.

Let $\pi$ be an element of $S^{*}$.

Consider the graph $\Gamma_{\pi}$ embedded in the projective plane $\pi$, which has vertices $V_{\pi}=V \cap \pi$ and edges $E_{\pi}$ which are the line segments induced from the lines of $L$ which are contained in $\pi$.  Euler's formula gives
$$
|V_{\pi}|-|E_{\pi}|+|F_{\pi}|=1,
$$
where $F_{\pi}$ is the number of faces of the embedded graph $\Gamma_{\pi}$.
Summing over the elements of $S^{*}$ gives
\begin{equation} \label{eq0}
\sum_{i \geqslant 3} i v_i- 2|E|+|F|=n,
\end{equation}
where $F$ is the set of faces of $\Gamma$ contained in a $\pi$ for some $\pi \in S^*$
and $v_i$ is the number of vertices incident with exactly $i$ planes of $S^{*}$.

Let $f_i$ denote the number of faces of $F$ with $i$ edges.
An edge is contained in a line $\ell$ which is the intersection of two planes of $S^{*}$. So each edge of $E$ is on precisely $4$ faces of $F$. Hence
\begin{equation} \label{eq1}
4|E|=\sum_{j \geqslant 3} j f_j.
\end{equation}

Next, we count in two ways pairs $(v,e)$ where $v$ is an end vertex of an edge $e$. If $v$ is a vertex incident with $i$ planes of $S^*$ then it is on $i(i-1)$ edges. Since an edge has two end vertices, we have
\begin{equation} \label{eq2}
2|E|=\sum_{i \geqslant 3} i(i-1) v_i.
\end{equation}

Combining (\ref{eq0}),  (\ref{eq1}) and  (\ref{eq2}), we obtain
$$
-\sum_{i \geqslant 3} i(i-4)v_i-\sum_{j \geqslant 4} (j-3)f_j=3n
$$
which we write as
\begin{equation} \label{eq3}
3v_3=3n+\sum_{i\geqslant 5} i(i-4)v_i+\sum_{j \geqslant 4} (j-3)f_j.
\end{equation}

Let us call an edge of $\Gamma$ {\em good} if it is an edge on four triangles of $F$ and if both its end vertices have degree $12$. If an edge is not good then it is {\em bad}.

\begin{lemma} \label{bad}
Let $S$ be a set of $n$ points of $\mathrm{PG}_3({\mathbb R})$, no three collinear and not all co-planar, spanning at most $Kn^2$ ordinary planes. Then the number of bad edges in $\Gamma$ is at most $30Kn^2$.
\end{lemma}

\begin{proof}
Since the number of planes incident with exactly $3$ points of $S$ is less than $Kn^2$, we have that $v_3 \leqslant Kn^2$. Then (\ref{eq3}) gives
$$
\sum_{j \geqslant 4} (j-3)f_j \leqslant 3Kn^2,
$$
which implies $\sum_{j \geqslant 4} f_j \leqslant 3Kn^2$ and so
$$
\sum_{j \geqslant 4} j f_j \leqslant 3Kn^2+9Kn^2=12Kn^2.
$$
Therefore, there are at most $12Kn^2$ bad edges which are the edge of some
non-triangular face.

Similarly (\ref{eq3}) implies
$$
\sum_{i\geqslant 5} i(i-4)v_i \leqslant 3Kn^2,
$$
which implies $\sum_{i\geqslant 5} i v_i \leqslant 3Kn^2$ and so
$$
\sum_{i\geqslant 5} i(i-1)v_i \leqslant 3Kn^2 + 3\sum_{i\geqslant 5} iv_i \leqslant 12Kn^2.
$$

The number of edges with end vertices are of degree $6$ or more than $12$ is at most
$$
6v_3+ \sum_{i \geqslant 5} i(i-1)v_i \leqslant 18Kn^2.
$$

The number of edges on faces (contained in some $\pi \in S^*$) which are not triangles is at most
$$
\sum_{j \geqslant 4} jf_j \leqslant 12Kn^2.
$$

Therefore, there are at most $30Kn^2$ bad edges.
\end{proof}

We call an edge of $\Gamma$ {\em rather good} if it is good and all edges emanating from its endpoints are also good. If an edge is not rather good then it is {\em slightly bad}.

\begin{lemma} \label{slightlybad}
Let $S$ be a set of $n$ points of $\mathrm{PG}_3({\mathbb R})$, no three collinear and not all co-planar, spanning at most $Kn^2$ ordinary planes. Then the number of slightly bad edges in $\Gamma$ is at most $690Kn^2$.
\end{lemma}

\begin{proof}
Each slightly bad edge is either bad or is a good edge which is adjacent to a bad edge. By Lemma~\ref{bad}, there are at most $30Kn^2$ bad edges. Each end of a good edge is a vertex of degree $12$, so we can bound the number of slightly bad edges in $\Gamma$ by
$$
30Kn^2+(22 \times 30)Kn^2=690Kn^2.
$$
\end{proof}

\section{The tetra-grid and the double-diamond}

We define a tetra-grid, which is a three-dimensional analogue of the triangular grid defined by Green and Tao in \cite{GT2013}.

Let $I$, $J$, $K$ and $L$ be four distinct intervals in ${\mathbb Z}$ (so $I$ is of the form $\{ i \in {\mathbb Z} \ | \ i_{-} \leqslant i \leqslant i_+\}$ for some $i_{-}$ and $i_+$). A {\em tetra-grid} is a collection of planes $(p_i^*)_{i \in I}$, $(q_j^*)_{j \in J}$, $(r_k^*)_{k \in K}$ and $(s_l^*)_{l \in L}$, (where $p_i^*$ indicates the plane dual to the point $p_i$), such that $i+j+k+l=0$ if and only if the planes $p_i^*,q_j^*,r_k^*,s_l^*$ meet in a point. The {\em dimension} of a tetra-grid is $|I| \times |J| \times |K| \times |L|$.

\begin{lemma} \label{notwisting}
Let $e$ be a rather good edge of $\Gamma$ lying on the line $p^*\cap q^*$, where $p,q \in S$. The sides of a triangle in $p^*$ with side $e$ has sides cut out by $r^*$ and $s^*$ if and only if the sides of a triangle in $q^*$ with side $e$ has sides cut out by $r^*$ and $s^*$.
\end{lemma}

\begin{proof}
Let us suppose that $e$ joins the vertices
$$
p_0^*\cap q_0^*  \cap r_{-1}^* \cap s_1^* \ \ \mathrm{and} \ \ \ p_0^*  \cap q_0^* \cap r_0^* \cap s_0 ^*,
$$ 
and that in the plane $q_0^*$, $e$ is a side of the triangles cut out by $r_0^*$ and $s_1^*$ on one side and $r_{-1}^*$ and $s_0^*$ on the other, see Figure~\ref{notwist}. Furthermore, suppose that $p_{-1}^*$ is the other plane of $S^*$ incident with $q_0^* \cap r_0^* \cap s_1^*$ and that $p_{1}^*$ is the other plane of $S^*$ incident with $q_0^* \cap r_{-1}^* \cap s_0^*$. Note that if $p_1=p_{-1}$ then $p_1 \cap q_0$ would intersect $e$ which it doesn't, so $p_1 \neq p_{-1}$.

\begin{figure}[h]
\centering
\includegraphics[width=4.5 in]{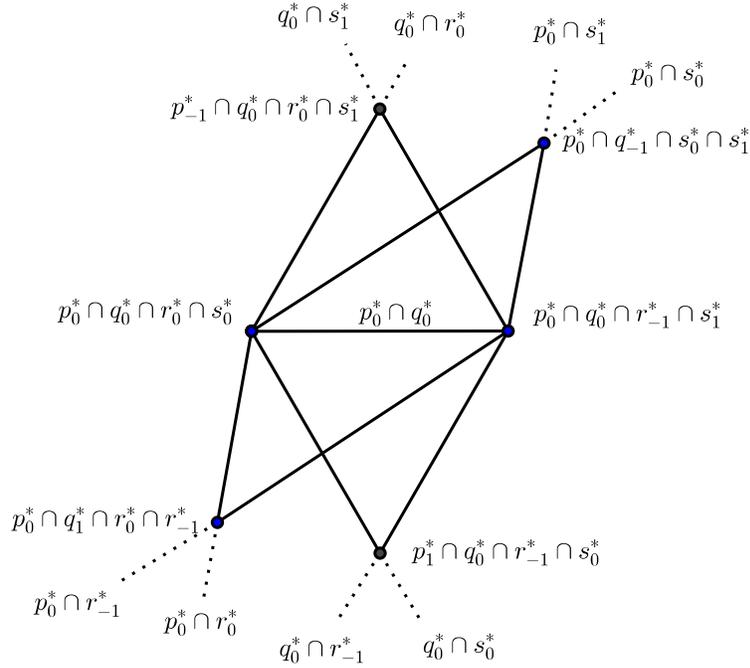}
\caption{The planes around an edge twisting.}
\label{notwist}
\end{figure}

We will prove by contradiction that $r_0^*$ and $s_1^*$ cut out a triangle in $p_0^*$ containing the edge $e$. Suppose that they don't, then without loss of generality we can assume that $r_0^*$ and $r_{-1}^*$ cut out one of the triangles containing $e$ as a side in $p_0^*$, and $s_0^*$ and $s_1^*$ cut out the other, see Figure~\ref{notwist}. Let $q_{1}^*$ be the other plane of $S^*$ incident with $p_0^* \cap r_0^* \cap r_{-1}^*$ and let $q_{-1}^*$ be the other plane of $S^*$ incident with $p_0^* \cap s_{1}^* \cap s_0^*$.

Consider the projection of the space from the point $p_0^*  \cap q_0^* \cap r_0^* \cap s_0 ^*$. The four planes $p_0^*$, $q_0^*$,  $r_0^*$ and  $s_0 ^*$ project onto lines. The plane $r_0^*$ does not cut $p_0^*$ between $q_0^*$ and $s_0^*$, so in the projection the line that $r_0^*$ projects onto, does not cut the projected line $p_0^*$ between $q_0^*$ and $s_0^*$. It does not cut $q_0^*$ between $p_0^*$ and $s_0^*$ either. Therefore, by Figure~\ref{0projection}, it doesn't cut $s_0^*$ between $q_0^*$ and $p_0^*$.

\begin{figure}[h]
\centering
\includegraphics[width=3.9 in]{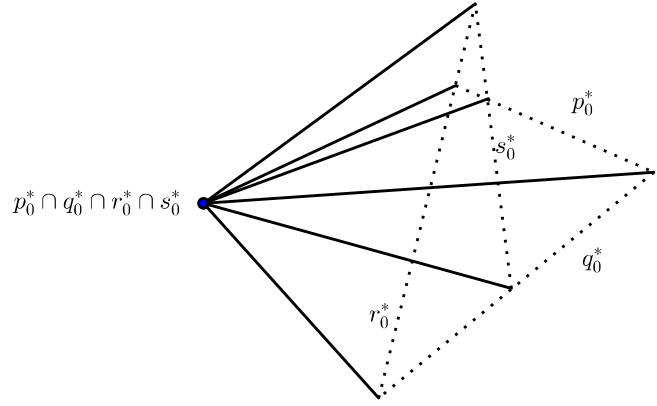}
\caption{The projection of {\rm Figure~\ref{notwist}} from $p_0^*  \cap q_0^* \cap r_0^* \cap s_0 ^*$.}
\label{0projection}
\end{figure}

This implies that the plane $r_0^*$ does not cut the plane $s_0^*$ in a line between the edge joining $p_0^*  \cap q_0^* \cap r_0^* \cap s_0 ^*$ to $p_0^*  \cap q_{-1}^* \cap s_0^* \cap s_1 ^*$ and 
the edge joining $p_0^*  \cap q_0^* \cap r_0^* \cap s_0 ^*$ to $p_1^*  \cap q_0^* \cap r_{-1}^* \cap s_0 ^*$.
Since the edges emanating from $p_0^*  \cap q_0^* \cap r_0^* \cap s_0 ^*$ are good, we deduce that the vertices $p_0^*  \cap q_0^* \cap r_0^* \cap s_0 ^*$, $p_0^*  \cap q_{-1}^* \cap s_0^* \cap s_1 ^*$ and $p_1^*  \cap q_0^* \cap r_{-1}^* \cap s_0 ^*$ form a triangle in $s_0^*$. Thus, $q_{-1}=p_1$.
By symmetry, we also get that $q_{-1}=p_{-1}$ and therefore $p_{-1}=p_1$, a contradiction.
\end{proof}

\begin{figure}[h]
\centering
\includegraphics[width=6 in]{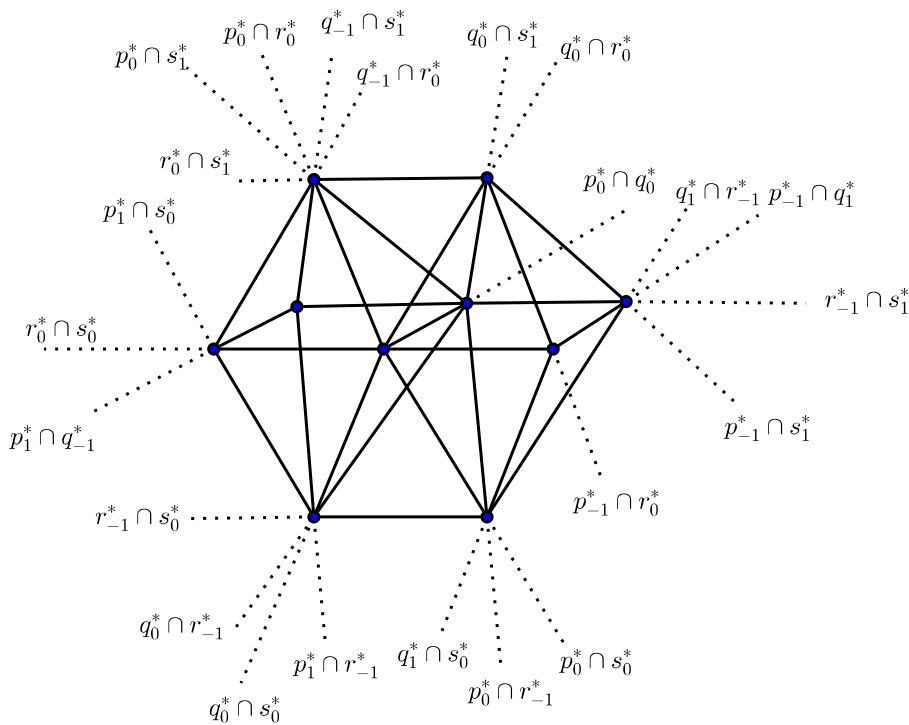}
\caption{The double diamond structure around a rather good edge.}
\label{doublediamond}
\end{figure}

\begin{lemma} \label{doubled}
The local structure of $\Gamma$ around a rather good edge is a tetra-grid containing a double diamond as indicated in Figure~\ref{doublediamond}.
\end{lemma}

\begin{proof}
Let us suppose that $e$ joins the vertices
$$
p_0^*\cap q_0^*  \cap r_{-1}^* \cap s_1^* \ \ \mathrm{and} \ \ \ p_0^*  \cap q_0^* \cap r_0^* \cap s_0 ^*,
$$ 
and that in the plane $q_0^*$, $e$ is a side of the triangles cut out by $r_0^*$ and $s_1^*$ on one side and $r_{-1}^*$ and $s_0^*$ on the other. Furthermore, label $p_{-1}^*$ as the other plane of $S^*$ incident with $q_0^* \cap r_0^* \cap s_1^*$ and $p_{1}^*$ as the other plane of $S^*$ incident with $q_0^* \cap r_{-1}^* \cap s_0^*$. 
\begin{figure}[h]
\centering
\includegraphics[width=3.6 in]{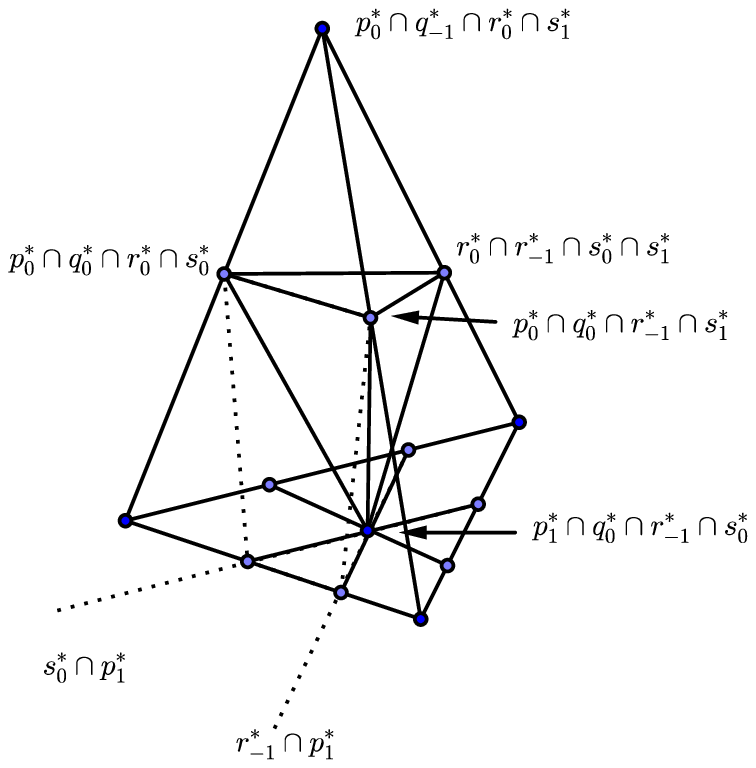}
\caption{The non-occuring diamond structure around a rather good edge.}
\label{5carat}
\end{figure}
By Lemma~\ref{notwisting}, in the plane $p_0^*$, $e$ is a side of the triangles cut out by $r_0^*$ and $s_1^*$ on one side and $r_{-1}^*$ and $s_0^*$ on the other. Let $q_{-1}^*$ be the other plane of $S^*$ incident with $p_0^* \cap r_0^* \cap s_1^*$ and let $q_{1}^*$ be the other plane of $S^*$ incident with $p_0^* \cap r_{-1}^* \cap s_0^*$.
\begin{figure}[h]
\centering
\includegraphics[width=3.6 in]{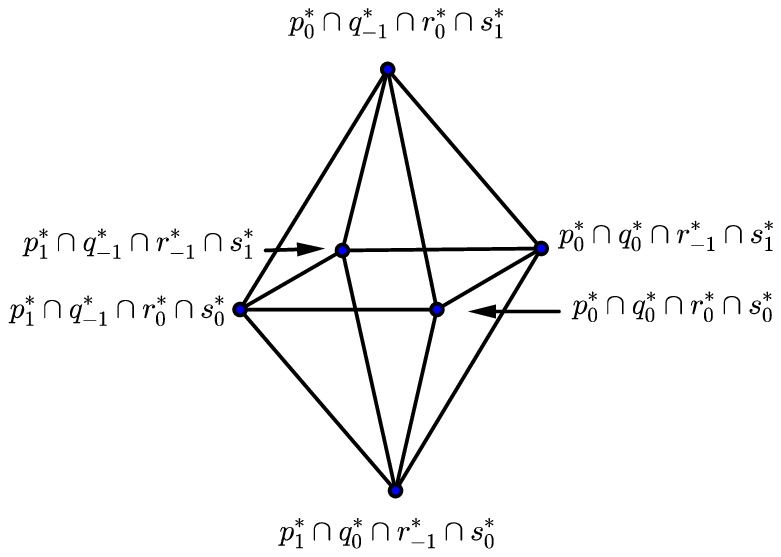}
\caption{The diamond structure around a rather good edge.}
\label{6carat}
\end{figure}
Consider the triangle in $r_0^*$ whose edges are cut out by $p_0^*$ and $s_0^*$ and the triangle in $s_1^*$ whose edges are cut out by $p_0^*$ and $r_{-1}^*$. These triangles are either closed by the same edge cut out by $r_0^* \cap s_1^*$, as in Figure~\ref{5carat}, or different edges which must be both cut out by $q_{-1}^*$, as in Figure~\ref{6carat}. In the former case, we have that $e$ is on a face in the plane $p_0^*$ which is not a triangle, which contradicts the fact that it is good. Hence, we conclude that we obtain a six vertex diamond structure containing the rather good edge $e$, as in Figure~\ref{6carat}, and by symmetry, the double diamond structure in Figure~\ref{doublediamond}.
\end{proof}

\section{The weak structure theorem}

 A key ingredient in the work of Green and Tao in \cite{GT2013} was Chasles' theorem, Theorem~\ref{8gives9}. 

\begin{theorem} \label{8gives9}
Suppose $\{ \ell_1,\ell_2,\ell_3\} $ and $\{ \ell_1^*,\ell_2^*,\ell_3^*\} $ are two sets of three lines and $\{ \ell_i \cap \ell_j^* \ | \ i,j \in \{1,2,3 \}\}$ define nine points of intersection in $\mathrm{PG}_2({\mathbb R})$. Then any cubic curve passing through eight of these points also passes through the ninth.
\end{theorem}

Here, the role of Theorem~\ref{8gives9} is played by Theorem~\ref{7gives8}. Theorem~\ref{7gives8} is a special case of the ``eight associated points theorem'', see \cite[Section 93.3]{Pedoe1988}. Since it plays an an important role in this article, we include a proof in the {\sc Appendix}.

\begin{theorem} \label{7gives8}
Suppose $\{ \pi_1,\pi_2 \} $, $\{ \pi_1',\pi_2' \} $  and $\{ \pi_1'',\pi_2'' \}$ are three sets of two planes and $\{ \pi_i \cap \pi_j' \cap \pi''_k \ | \ i,j,k \in \{1,2 \}\}$ define eight points of intersection in $\mathrm{PG}_3({\mathbb R})$. Then any quadric which passes through seven of these points also passes through the eighth.
\end{theorem}

In the following lemma we apply Theorem~\ref{7gives8} to the double diamond structure in Lemma~\ref{doubled}. The double diamond is cut out around a rather good edge by ten planes in the dual space. In Figure~\ref{doublediamond}, the ten points dual to these planes are $\{p_{-1},p_0,p_1,q_{-1},q_0,q_1,r_{-1},r_0,s_0,s_1\}$.

\begin{lemma} \label{tenpoints}
The ten points of $S$, which in the dual space cut out the double diamond structure from Lemma~\ref{doubled} around a rather good edge of $\Gamma$, are contained in the intersection of two linearly independent quadrics. 
\end{lemma}

\begin{proof}
Suppose that $\{p_{-1},p_0,p_1,q_{-1},q_0,q_1,r_{-1},r_0,s_0,s_1\}$ are the ten points of $S$ that in the dual space are the ten planes which cut out the double diamond structure around the rather good edge $e$, as in Figure~\ref{doublediamond}. Define six planes $\pi_1$, $\pi_2$, $\pi_1'$, $\pi_2'$, $\pi_1''$, $\pi_2''$ to be the planes which contain the four points, as indicated in the following array.
$$
\begin{array}{||c|c|c|c||}
\hline
\pi_1 & p_0,q_0,r_0,s_0 & \pi_2 & p_1,q_{-1},r_{-1},s_1 \\
\pi_1' & p_0,q_{0},r_{-1},s_1 & \pi_2' & p_1,q_{-1},r_{0},s_0\\
\pi_1''& p_0,q_{-1},r_{0},s_1 & \pi_2'' & p_1,q_{0},r_{-1},s_0\\
\hline
\end{array}
$$

The vector space of homogeneous polynomials of degree two in four variables has dimension $10$. Therefore, there are two linearly independent quadrics containing the eight points $p_{-1},p_0,p_1,q_0,q_{-1},s_0,r_{-1},r_0$. By Theorem~\ref{7gives8}, these quadrics also contain the point $s_1$.

Now redefine the six planes $\pi_1$, $\pi_2$, $\pi_1'$, $\pi_2'$, $\pi_1''$, $\pi_2''$, by the four points they contain, as in the following array.
$$
\begin{array}{||c|c|c|c||}
\hline
\pi_1 & p_0,q_0,r_0,s_0 & \pi_2 & p_{-1},q_{1},r_{-1},s_1 \\
\pi_1' & p_0,q_{0},r_{-1},s_1 & \pi_2' & p_{-1},q_{1},r_{0},s_0\\
\pi_1''& p_0,q_{1},r_{-1},s_0 & \pi_2'' & p_{-1},q_{0},r_{0},s_1\\
\hline
\end{array}
$$

Applying Theorem~\ref{7gives8} to these re-defined planes, we conclude that the quadrics also contain the point $q_1$.

\end{proof}

\begin{lemma} \label{segment}
Suppose that $p^* \cap q^*$ contains a segment $T$ of $m$ rather good edges in $\Gamma$. Then there are two linearly independent quadrics both containing $p$ and $q$ with the property that if $r^*$ intersects a vertex of $T$, for some $r \in S$, then $r$ is also a point of both quadrics.\end{lemma}

\begin{proof}
Let $p_0=p$ and $q_0=q$ and label the vertices in the segment of rather good edges as
$p_0^* \cap q_0^* \cap r_{-j}^* \cap s_j^*$, for $j=0,\ldots,m$.

By Lemma~\ref{doubled}, we get a tetra-grid of dimension $3 \times 3 \times m\times m$ as indicated in Figure~\ref{diamonds}. Note that in Figure~\ref{diamonds}, only half of the double diamond structure is drawn. 
\begin{figure}[h]
\centering
\includegraphics[width=6.6 in]{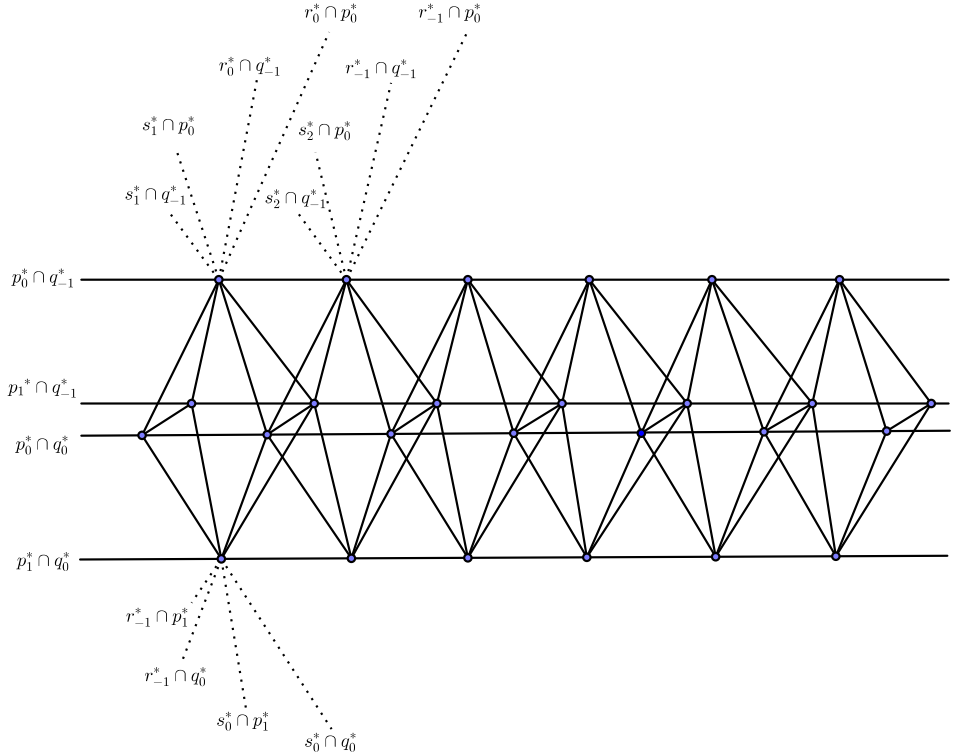}
\caption{The tetra-grid around a segment of rather good edges.}
\label{diamonds}
\end{figure}
By Lemma~\ref{tenpoints}, there are two linearly independent quadrics $\psi_1$ and $\psi_2$ containing the points
$$
p_{-1},p_0,p_1,q_{-1},q_0,q_{1},r_{-1},r_0,s_0,s_1.
$$
Now apply Theorem~\ref{7gives8} with the six planes $\pi_1$, $\pi_2$, $\pi_1'$, $\pi_2'$, $\pi_1''$, $\pi_2''$, defined by the four points they contain as in the following array.
$$
\begin{array}{||c|c|c|c||}
\hline
\pi_1 & p_0,q_1,r_{-2},s_1 & \pi_2 & p_{1},q_{0},r_{-1},s_0 \\
\pi_1' & p_0,q_{1},r_{-1},s_0 & \pi_2' & p_{1},q_{0},r_{-2},s_1\\
\pi_1''& p_0,q_{0},r_{-1},s_1 & \pi_2'' & p_{1},q_{1},r_{-2},s_0\\
\hline
\end{array}
$$
Again the quadrics $\psi_1$ and $\psi_2$ contain seven of the points of intersection, so they contain the eighth point $r_{-2}$ too.
Now, redefine the planes as in the following array.
$$
\begin{array}{||c|c|c|c||}
\hline
\pi_1 & p_0,q_0,r_{-2},s_2 & \pi_2 & p_{-1},q_{1},r_{-1},s_1 \\
\pi_1' & p_0,q_{0},r_{-1},s_1 & \pi_2' & p_{-1},q_{1},r_{-2},s_2\\
\pi_1''& p_0,q_{1},r_{-2},s_1 & \pi_2'' & p_{-1},q_{0},r_{-1},s_2\\
\hline
\end{array}
$$

Since the quadrics $\psi_1$ and $\psi_2$ contain seven of the points of intersection, they contain the eighth point $s_{2}$ too. Continuing in this way we conclude that the quadrics contain all the points $r_{-j}$ and $s_j$, for $j=1,\ldots,m$.
\end{proof}

In the following theorem we make no restriction on $K$, so Theorem~\ref{weak} gives some information for any set of $n$ points, no 3 collinear, spanning $o(n^3)$ ordinary planes. In the following sections we will bound $K$ which will allow us to obtain stronger structural results.

\begin{theorem} \label{weak} {\rm (Weak structure theorem)}
Let $S$ be a set of $n$ points in $\mathrm{PG}_3({\mathbb R})$, no three collinear, and not all co-planar. If there are less than $Kn^2$ ordinary planes, $K \geqslant 1$, then $S$ is contained in the union of $4143K$ varieties, each of which is the intersection of two quadrics.
\end{theorem}

\begin{proof}
By Lemma~\ref{slightlybad}, there are at most $690Kn^2$ slightly bad edges in $\Gamma$.

By the pigeon-hole principle, there are points $p,q \in S$, such that $p^*\cap q^*$ contains at most $1380K/(1-1/n)$ slightly bad edges. 
We can assume that $n>4143K$ since otherwise the claim is trivial, and conclude that $p^*\cap q^*$ contains at most $1381K$ slightly bad edges. 

The slightly bad edges split the edges of $\Gamma$ along $p^*\cap q^*$ into segments of rather good edges and vertices incident with slightly bad edges. By Lemma~\ref{segment}, the points whose dual plane is incident with a vertex in the same segment of rather good edges all lie in the same quadric intersection variety. There are at most $1381K$ such segments.

Suppose that $v_0,\ldots,v_{m-1}$ are the vertices on $p^* \cap q^*$ which are the end vertices of some slightly bad edge. Define linear forms $\beta_j$ so that $v_j^*=\ker \beta_j$. Then $v_j$ is on the intersection of plane pair quadrics $\beta_{j-1}\beta_j$ and $\beta_j\beta_{j+1}$, for $j=0,\ldots,m$ (indices read modulo $m$). Therefore, the points of $S$ whose dual plane is incident with a vertex of a slightly bad edge on $p^* \cap q^*$, are contained in the union of $2762K$ varieties, which are the intersection of plane pair quadrics.
\end{proof}

\section{Some lemmas on quadrics and cubics}

In this section we will prove some lemmas on quadrics and cubics which we will need in the next section to prove the main structure theorem.

For any set of homogeneous polynomials $f_1,\ldots,f_r$ in four variables, we denote by
$$
V(f_1,\ldots,f_r),
$$
the set of common zeros of $f_1,\ldots,f_r$, a subset of $\mathrm{PG}_{3}({\mathbb R})$.

Let $\psi_1,\psi_2$ be quadratic forms. Denote by $b_i(X,Y)$ the symmetric bilinear form obtained from the polarisation of $\psi_i$, in other words
$$
b_i(X,Y)=\psi_i(X+Y)-\psi_i(X)-\psi_i(Y),
$$
for $i=1,2$.

For each $p \in V(\psi_1,\psi_2)$, let 
$$
\phi_p(X)=b_1(p,X)\psi_2(X)-b_2(p,X)\psi_1(X).
$$

\begin{lemma} \label{phip}
Let $q \in V(\psi_1,\psi_2)$, for some quadratic forms $\psi_1,\psi_2$. If $y \in V(\phi_q)$ then the line joining $q$ to $y$ is contained in $V(\phi_q)$.
\end{lemma}

\begin{proof}
We have to show that
$$
\phi_q(y+\lambda q)=0,
$$
for all $\lambda \in {\mathbb R}$, whenever $y \in V(\phi_q)$.

By direct calculation, observing that $b_1(q,q)=b_2(q,q)=0$,
$$
\phi_q(y+\lambda q)=b_1(q,y+\lambda q)\psi_2(y+\lambda q)-b_2(q,y+\lambda q)\psi_1(y+\lambda q)
$$
$$
=b_1(q,y)(\psi_2(y)+b_2(y,\lambda q))-b_2(q,y)(\psi_1(y)+b_1(y,\lambda q))=\phi_q(y)=0.
$$
\end{proof}

Lemma~\ref{phip} implies that $V(\phi_q)$ is a degenerate cubic surface which is a cone with vertex $q$ of a cubic curve in the projective plane obtained by projecting the space from $q$.

Let $\phi$ be a homogeneous polynomial in four variables. The tangent plane at a point $p$ of the surface $V(\phi)$, is
$$
\sum_{i=1}^4 \frac{d\phi}{dX_i} (p)X_i.
$$

Let $q$ be a point of $V(\psi_1,\psi_2)$ and suppose that $p\in V(\phi_q)$. We denote by  $t_p^q(X)$, the tangent plane to $V(\phi_q)$ at $p$.

Let $p$ be a point of $V(\psi_1,\psi_2)$ for which $\ker b_1(p,X) \neq \ker b_2(p,X)$ and define a line $\ell_p=\ker b_1(p,X) \cap \ker b_2(p,X)$.

\begin{lemma} \label{tangentproject}
If $p,q \in V(\psi_1,\psi_2)$ and $\ker b_1(p,X) \neq \ker b_2(p,X)$ then the line $\ell_p$ is contained in $t_p^q(X)$ and $V(\phi_p)$.
\end{lemma}

\begin{proof}
Since $\psi_i$ is a quadratic form, the tangent plane to the quadric $V(\psi_i)$ at $p$ is $\ker b_i(p,x)$. The tangent plane at the point $p$ is
$$
t_p^q(X)=\sum_i \frac{d \phi_q}{dX_i}(p)X_i=\sum_i b_1(q,p)\frac{d \psi_2}{dX_i}(p)X_i-b_2(q,p)\frac{d \psi_1}{dX_i}(p)X_i
$$
$$
+\psi_2(p)\frac{d b_1(q,X)}{dX_i}(p)X_i-\psi_1(p)\frac{d b_2(q,X)}{dX_i}(p)X_i
$$
$$
=b_1(q,p) b_2(p,X)-b_2(q,p)b_1(p,X)+\psi_2(p)b_1(q,X)-\psi_1(p) b_2(q,X).
$$
Since $p$ is a point of $V(\psi_1,\psi_2)$, we have that $t_p^q(x)=0$, for all $x \in \ell_p$.

It is immediate from the definitions that $\ell_p$ is contained in $V(\phi_p)$.

\end{proof}

Define
$$
\phi_{pq}(X)=b_1(p,X)b_2(q,X)-b_1(q,X)b_2(p,X).
$$

\begin{lemma} \label{thisorthat}
Suppose $p,q \in V(\psi_1,\psi_2)$. If $r \in V(\phi_p,\phi_q)$ and $r \not\in V(\phi_{pq})$ then $r \in V(\psi_1,\psi_2)$.
\end{lemma}

\begin{proof}
Since $r \in V(\phi_p,\phi_q)$, we have
$$
b_1(p,r)\psi_2(r)-b_2(p,r)\psi_1(r)=0
$$
and
$$
b_1(q,r)\psi_2(r)-b_2(q,r)\psi_1(r)=0.
$$
Therefore, $\phi_{pq}(r)\psi_2(r)=0$ and $\phi_{pq}(r)\psi_1(r)=0$.
\end{proof}

We will need the following lemmas in the proof of Lemma~\ref{typeone}.

A bilinear form $b(X,Y)$ is {\em degenerate} at some point $u$, if $b(X,u)=0$. If the polariation of a quadratic form $\phi$ is degenerate at some point $u$, then we say that the quadratic form $\phi$ is also {\em degenerate} at $u$.

\begin{lemma} \label{2degen}
Let $b$ be the polarisation of a quadratic form $\psi$. If $b(X,Y)$ is degenerate at two points, $u$ and $v$ say, and $w$ is a point of $V(\psi)$, not on the line spanned by $u$ and $v$, then the plane spanned by $u$, $v$ and $w$ is in $V(\psi)$. Hence $V(\psi)$ consists of two planes or one repeated plane.
\end{lemma}

\begin{proof}
This is direct calculation. For all $\lambda, \mu \in {\mathbb R}$,
$$
\psi(w+\lambda u + \mu v)=b(w,\lambda u + \mu v)+\psi(w)+\psi(\lambda u+\mu v)=\lambda b(w,u)+\mu b(w,v)=0.
$$
\end{proof}

\begin{lemma} \label{theRlemma}
Let $\psi_1$ and $\psi_2$ be linearly independent quadratic forms, whose polarisations are $b_1$ and $b_2$ respectively. If $R$ is a subset of $V(\psi_1,\psi_2)$ spanning the whole space, with the property that there are at least five points, no three collinear, in $U=\cap_{r,q \in R} V(\phi_{rq})$ then one of the following holds.
\begin{enumerate}
\item[(i)]
There is a point $s$ where both $b_1$ and $b_2$ are degenerate.
\item[(ii)]
$R$ is contained in two (possible degenerate) planar conics.
\item[(iii)]
$U$ is contained in a conic.
\end{enumerate}
\end{lemma}

\begin{proof}
Suppose that $b_2$ is degenerate at two points. By Lemma~\ref{2degen}, $V(\psi_2)$ consists of two (possibly identical) planes $\pi_1$ and $\pi_2$. Since $R$ is contained in $V(\psi_2)$ and spans the whole space, $\pi_1 \neq \pi_2$. Since $R$ is contained in $V(\psi_1,\psi_2)$ and spans the whole space, $V(\psi_1)$ does not intersect $V(\psi_2)$ in $\pi_1$ or $\pi_2$. Hence $V(\psi_1)$ intersects $\pi_i$ in a (possibly degenerate) conic. Since $R$ is contained in $V(\psi_1,\psi_2)$, we have (ii).

We will assume from now on that $b_2$ is degenerate at at most one point.

Suppose there is a point $s \in U$ for which $\ker b_1(X,s) \neq \ker b_2(X,s)$. Since $s \in U$,
$$
b_1(r,s)b_2(q,s)-b_1(q,s)b_2(r,s)=0,
$$
for all $q,r \in R$. Since $R$ spans the whole space, we can find $\lambda_r \in R$ such that
$$
\sum_r \lambda_r r \in \ker b_2(X,s) \setminus \ker b_1(X,s).
$$
Summing the above accordingly, we have that
$$
b_1(\sum_r \lambda_r r,s) b_2(q,s)=0,
$$
and $b_1(\sum_r \lambda_r r,s) \neq 0$, and so $b_2(q,s)=0$, for all $q \in R$. This implies that $b_2(X,s)=0$ and so $b_2$ is degenerate at $s$. Similarly $b_1$ is degenerate at $s$ as well, and we have (i).

Suppose $\ker b_1(X,s) = \ker b_2(X,s)$, for each of the five points $s \in \{e_1,e_2,e_3,e_4,u \} \subset U$. Let $k$ be such that $e_1,\ldots,e_k$ are linearly independent and that $u=\sum_{i=1}^ku_i e_i$, for some $u_i \in {\mathbb R} \setminus \{ 0 \}$. Since, by assumption the points $e_1,\ldots,e_k,u$ are not collinear, $k \geqslant 3$. 

We have that for $i=1,\ldots,k$, there is a $\lambda_i \in {\mathbb R}$ such that $b_1(X,e_i)=\lambda_i b_2(X,e_i)$ and that there is a $\lambda_u$ such that $b_1(X,u)=\lambda_u b_2(X,u)$.
Substituting $u=\sum_{i=1}^ku_i e_i$ and using $b_1(X,e_i)=\lambda_i b_2(X,e_i)$, we get
$$
\sum_{i=1}^k u_i(\lambda_i-\lambda_u)b_2(X,e_i)=0.
$$

If $b_2$ is non-degenerate then 
$$
\sum_{i=1}^k u_i(\lambda_i-\lambda_u) e_i=0.
$$
Therefore $\lambda_i=\lambda_u=\lambda$ for $i=1,\ldots,k$, since $e_1,\ldots,e_k$ are linearly independent. This implies that $(b_1-\lambda b_2)(X,y)=0$ for all $y$ in the subspace spanned by $e_1,\ldots,e_k$. 
If $k=4 $ then $b_1 = \lambda b_2$ which implies $\psi_1$ and $\psi_2$ are linearly dependent, which by assumption, they are not. Therefore, $k=3$ and
$(b_1-\lambda b_2)(X,y)=0$ for all $y$ in the plane spanned by $e_1,e_2,e_3$. Hence $(\psi_1-\lambda \psi_2)(y)=0$ for all $y$ in the plane spanned by $e_1,e_2,e_3$.
Let $\psi=\psi_1-\lambda \psi_2$ and let $\psi'=\psi_1-\lambda' \psi_2 \neq \psi$, for some $\lambda' \in {\mathbb R}$. We have shown that $V(\psi)$ consists of two (possibly identical) planes $\pi_1$ and $\pi_2$. Since $R \in V(\psi_1,\psi_2)$, it follows that $R \in V(\psi,\psi')$. Since $R$ spans the whole space, $\pi_1 \neq \pi_2$ and $V(\psi')$ intersects $\pi_i$ in a (possible degenerate) conic, so we have (ii).

If $b_2$ is degenerate at $s=\sum_{i=1}^k u_i(\lambda_i-\lambda_u) e_i$, then, since $b_1(X,e_i)=\lambda_i b_2(X,e_i)$, we have that $b_1(s,e_i)=0$. If $k=4$ then this implies that $b_1$ is also degenerate at $s$, so we have (i) again. This leaves the case $k=3$. We are in the case that $b_2$ is degenerate at at most one point, so we have that for all $v \in U \setminus \{e_1,e_2,e_3,u \}$,
$\sum_{i=1}^3 v_i(\lambda_i-\lambda_v) e_i$ is the point $s$. Therefore, 
$$
\frac{v_3(\lambda_3-\lambda_v)}{v_2(\lambda_2-\lambda_v)}=\frac{u_3(\lambda_3-\lambda_u)}{u_2(\lambda_2-\lambda_u)}
$$
and 
$$
\frac{v_3(\lambda_3-\lambda_v)}{v_1(\lambda_1-\lambda_v)}=\frac{u_3(\lambda_3-\lambda_u)}{u_1(\lambda_1-\lambda_u)}.
$$
Eliminating $\lambda_v$ from these two equation, it follows that $v$ is a zero of the quadratic form, 
$$
X_1X_2u_3(\lambda_3-\lambda_u)(\lambda_2-\lambda_1)+X_1X_3u_2(\lambda_2-\lambda_u)(\lambda_1-\lambda_3)+X_2X_3u_1(\lambda_1-\lambda_u)(\lambda_3-\lambda_2),
$$
which is also zero at the points of $\{e_1,e_2,e_3,u \}$. Thus, we have (iii).
\end{proof}

\begin{lemma} \label{2pis}
Let $\psi_1$ and $\psi_2$ be linearly independent quadratic forms, where $V(\psi_2)$ consists of two planes and $V(\psi_1)$ intersects $V(\psi_2)$ in two non-degenerate conics. Then there are at most two quadratic forms, in the subspace spanned by $\psi_1$ and $\psi_2$, distinct from $\psi_2$, which are degenerate.
\end{lemma}

\begin{proof}
After a suitable projective transformation we can assume that $\psi_2(X)=X_3X_4$.

Let $\psi$ be a quadratic form in the subspace spanned by $\psi_1$ and $\psi_2$, which is degenerate at the point $x$.

Suppose that for some $\lambda \in {\mathbb R}$, the quadratic form $\psi(X)+\lambda X_3X_4$ is degenerate at the point $y$.

We will prove that the point $y$ is determined.

Let $b=\sum_{i=1}^4 \alpha_i(Y)X_i$ be the polarisation of $\psi$. Then, since $\psi$ is degenerate at the point $x$,
$$
\sum_{i=1}^4 \alpha_i(Y)x_i=0,
$$
and $\alpha_i(x)=0$ for $i=1,\ldots,4$.

The polarisation of the quadratic form $\psi(X)+\lambda X_3X_4$ is 
$$
b(X,Y)+\lambda(X_3Y_4+Y_3X_4).
$$
Since $b$ is degenerate at $y$,
$$
0=b(X,y)+\lambda(X_3y_4+y_3X_4),
$$
so $\alpha_1(y)=\alpha_2(y)=0$ and $\alpha_3(y)+\lambda y_4=0$ and $\alpha_4(y)+\lambda y_3=0$.

Now, $b(x,y)=0$ implies $x_3 \alpha_3(y)+x_4 \alpha_4(y)=0$, so combining these equations, we get $(x_4y_3+y_4x_3)\alpha_3(y)=0$. The linear forms $\alpha_1,\alpha_2,\alpha_3$ and $\alpha_4$ are linearly dependent (since $x$ is a common zero of these forms), so $\alpha_3(y)=0$ implies $x=y$. Hence, $y$ is a common zero of $\alpha_1(y)=\alpha_2(y)=0$ and $x_4y_3+y_4x_3=0$.

If $\alpha_1$ and $\alpha_2$ are linearly dependent or $x_3=x_4=0$ then $\psi$ is degenerate at a point $(1,a,0,0)$, which is one the line $\ker X_3 \cap \ker X_4$. This implies that  $V(\psi)$ intersect $X_3$ and $X_4$ in degenerate conics, which contradicts the hypothesis.

If $x_3=0$ and $x_4 \neq 0$ then $\alpha_4(y)=0$ and so $x=y$. Likewise, if $x_4=0$ and $x_3 \neq 0$ then $\alpha_3(y)=0$ and so $x=y$.

If $x_3x_4 \neq 0$ then $x_3X_4+x_4X_3$ is a linear form, linearly independent from $\alpha_1(X)$ and $\alpha_2(X)$, since $\alpha_1(x)=\alpha_2(x)=0$ and $2x_3x_4 \neq 0$. So we conclude that there is a unique point $y$ which is a common solution of $\alpha_1(y)=\alpha_2(y)=0$ and $x_4y_3+y_4x_3=0$.
\end{proof}

\section{The main structure theorem}

In the following three lemmas, we will make extensive use of the following theorem, which is \cite[Proposition 5.3]{GT2013}.

\begin{theorem} \label{GTinter}
Let $S$ be a set of $n$ points in $\mathrm{PG}_2({\mathbb R})$ spanning at most $Kn$ ordinary lines, for some $K\geqslant 1$. Then one of the following holds:
\begin{enumerate}
\item[(i)]
$S$ is contained in the union of an irreducible cubic and an additional $2^{75}K^5$ points.
\item[(ii)]
$S$ is contained in the union of an irreducible conic $\sigma$ and an additional $2^{64}K^4$ lines. Furthermore, $\sigma$ contains between $\frac{1}{2}n-2^{76}K^5$ and $\frac{1}{2}n+2^{76}K^5$ points of $S$, and $S \setminus \sigma$ spans at most $2^{62}K^4n$ ordinary lines.
\item[(iii)]
$S$ is contained in the union of $2^{16}K$ lines and an additional $2^{87}K^6$ points.
\end{enumerate}
\end{theorem}

In the first lemma of its type, we consider the case that $S$ has many points which project the other points of $S$ onto a set of type (i) of Theorem~\ref{GTinter}.

\begin{lemma} \label{typeone} 
Let $S$ be a set of $n$ points in $\mathrm{PG}_3({\mathbb R})$, no three collinear and not all co-planar, spanning less than $Kn^2$ ordinary planes, for some function $K=o(n^{\frac{1}{6}})$. Let $S'$ be the subset of all points of $S$ which project $S$ onto a set spanning at most $dKn$ ordinary lines, for some large constant $d$. Suppose there is a subset $S''$ of $S'$ of size at least $(1-4/d)n$, where the projection from each point $p$ of $S''$ is contained in the union of an irreducible cubic $V(\psi_p)$ and an additional $2^{75}(dK)^5$ points. Then all but at most $O(K)$ points of $S$ are contained in the intersection of two quadrics. Furthermore, all but at most $O(K)$ points of $S'$ are incident with at least $\frac{3}{2}n-O(K)$ ordinary planes.
\end{lemma}

\begin{proof}
This proof is rather long, so we break it down into three claims.

Claim 1: The projection of $S$ from $p \in S''$ is contained in a coset $C_p$ of a subgroup of the non-singular points of the curve $V(\psi_p)$ and an additional $O(K)$ points.

Let $p \in S''$. By hypothesis, the projection of $S$ from $p$ is contained in an irreducible cubic $V(\psi_p)$ and an additional $2^{75}(dK)^5$ points and spans at most $dKn$ ordinary lines. By \cite[Lemma 7.2]{GT2013}, $V(\psi_p)$ is either an elliptic curve or an acnodal cubic and the projection differs from a coset $C_p=H_p \oplus x_p$, where $H_p$ is a subgroup of the non-singular points of the curve and $x_p\oplus x_p\oplus x_p \in H_p$, by at most $\epsilon=O((dK)^5)$ points. If $q$ is a point of the projection that is not in $V(\psi_p)$ then by \cite[Lemma 7.7]{GT2013}, $q$ is incident with at least $n/1000-\epsilon$ ordinary lines. If $q$ is a point of the projection that is in $V(\psi_p) \setminus C_p$ then $q$ is incident with at least $n-2\epsilon$ ordinary lines, since a line joining $q$ to a point of $C_p$ will be ordinary unless it is incident with another point of the projection of $S$ from $p$, which is necessarily not in $C_p$. Thus, we have that $\epsilon(n/1000-\epsilon) \leqslant 2dKn$ and therefore $\epsilon=O(K)$, since $\epsilon=o(n)$. 

Claim 2: All but at most $O(K)$ points of $S$ are contained in the intersection of two quadrics.

Define $(q)_p$ to be the projection of the point $q$ from $p$.

Let $p \in S''$. Firstly, we show that $\Gamma$ has a lot of edges. Let $(q)_p$ be a point of the coset $C_p$. At least $n-O(K)$ points of the projection are contained in $C_p$, so there are $\frac{1}{2}(n-O(K))$ lines incident with $(q)_p$ and a further two points of the projection. Therefore, there are at least  $\frac{1}{2}(n-O(K))$ planes containing the line joining $p$ and $q$ which contain at least $4$ points of $S$. This implies that in the graph $\Gamma$ the line $p^* \cap q^*$ contains at least $\frac{1}{2} (n-O(K))$ edges. Given that there are at least $\frac{1}{2}(n-O(K))(1-4/d)n$ such pairs $p$ and $q$ we deduce that, for $n$ and $d$ large enough, the graph $\Gamma$ has more than $\frac{1}{5}n^3$ edges. 

Let us call all the edges on a line $p^* \cap q^*$, where $p \in S$ and $q \in S\setminus S''$, {\em rotten}. There are less than $4n^3/d$ rotten edges. Let us call an edge {\em really good} if it is good, not rotten, and all the edges emanating from its end vertices are also good and not rotten. Recall that a good edge is an edge both of whose end vertices have degree twelve and which is on four triangles. If an edge is not really good, then either it is bad or rotten or it is good and not rotten and one of the edges emanating is bad or rotten. We can bound the number of not really good edges by
$$
30Kn^2+(4/d)n^3+(2 \times 11)(30Kn^2+(4/d)n^3)<92n^3/d+O(Kn^2)<93n^3/d,
$$
for $d$ large enough.

By the pigeon-hole principle, there is a $p \in S''$, such that $p^*$ has at most 
$$
2 \times \frac{93n^3}{d}/(1-\frac{4}{d})n= \frac{186n^2}{d-4}
$$
not really good edges. 

There are at least $n-4(n/d)-cK$ points of $S''$ which are projected from $p$ onto $C_p$, for some constant $c$ and $n$ large enough.
Again by the pigeon-hole principle, there is a $q \in S''$, projected from $p$ onto $C_p$, such that $p^* \cap q^*$ has at most 
$$
\frac{186n}{(d-4)(1-4/d-cK/n)}
$$ 
not really good edges. By Claim 1, there are at most $O(K)$ points of $S$ which are not projected from $p$ onto $V(\psi_p)$, and likewise at most $O(K)$ points of $S$ which are not projected from $q$ onto $V(\psi_q)$. If $r$ is such a point then we also count the edges emanating from $p^* \cap q^* \cap r^*$ as not really good. This discards at most $O(K)$ extra edges. Therefore, there are at most
$$
\frac{186n}{(d-4)(1-4/d-cK/n)}+O(K)
$$ 
not really good edges on the line $p^* \cap q^*$.
Since the line $p^* \cap q^*$ has at least $\frac{1}{2} (n-O(K))>\frac{1}{3}n$ edges, we can conclude that there is a segment $T$ on $p^* \cap q^*$ of at least 13 really good edges, since
$$
\frac{1}{558}(d-4)(1-4/d-cK/n)>13,
$$
choosing $d$ large enough. 

Let $R$ be a set of $13$ points of $S$ where $r \in R$ if and only if $r^*$ cuts $p^* \cap q^*$
in the segment $T$. Since $T$ is a segment of really good edges $R$  is a subset of $S''$. Since a really good edge is a rather good edge, Lemma~\ref{segment} implies that there are two linearly independent quadrics $V(\psi_1)$ and $V(\psi_2)$ both containing $R \cup \{p,q\}$. The points of $R$ are projected from $p$ onto both the projection of $V(\phi_p)$, where $\phi_p$ is defined (as before) as 
$$
\phi_p(x)=b_1(p,x)\psi_2(x)-b_2(p,x)\psi_1(x),
$$
and $V(\psi_p)$. By Lemma~\ref{phip}, the projection of $V(\phi_p)$ from $p$ is a cubic curve. Since $|R|>9$  and $V(\psi_p)$ is an irreducible cubic, we have by Bezout's theorem that the two curves are the same. Similarly, the projection of $V(\phi_q)$ from $q$ and $V(\psi_q)$ are the same curve.

Let $S_{pq}$ be the set of points of $S$ which are projected onto $V(\psi_q)$ from $q$ and onto $V(\psi_p)$ from $p$. By Claim 1, $S_{pq}$ contains at least $n-O(K)$ points of $S$. By the previous paragraph the points of $S_{pq}$ are contained in $V(\phi_p,\phi_q)$ and so by Lemma~\ref{thisorthat}, if a point of $S_{pq}$ is not in $V(\psi_1,\psi_2)$, then it is in $V(\phi_{pq})$.

Now we consider two cases. 

Suppose that for all $r \in R$, the irreducible cubic $V(\psi_r)$, from Claim 1, is the projection of $V(\phi_r)$. Then, repeating the above, we have that there are $n-O(K)$ points of $S$ which are in $V(\phi_{rt})$, for all $r,t \in R\cup \{p,q\}$ or they are in $V(\psi_1,\psi_2)$. Let
$$
U=\cap_{r,t \in R\cup \{p,q\}} V(\phi_{rt}).
$$
The set $R\cup \{p,q\}$ spans the whole space, so Lemma~\ref{theRlemma} applies. The set $R$ is not contained in two planar conics, since the projection of $R$ from $p$ is onto an irreducible cubic, which by Bezout's theorem intersects the union of two conics in at most $12$ points, so we are in case (i) or (iii). If we are in case (i) then $b_1$ and $b_2$ are both degenerate at the same point $s$ and so for all $x \in V(\phi_p)$ and all $\lambda \in {\mathbb R}$, we have $\phi_p(x+\lambda s)=0$, which implies that the projection of $V(\phi_p)$ contains lines. However, the projection of $V(\phi_p)$ is $V(\psi_p)$ which is an irreducible cubic, so this case does not occur. If we are in case (iii) then $U$ is contained in a conic. By Bezout's theorem at most six points of $U$ are projected from $p$ onto $C_p$. Hence, by Claim 1, $|U|=O(K)$. This implies that at least $n-O(K)$ points of $S$ are contained in $V(\psi_1,\psi_2)$. 

Now, suppose that there is an $r \in R$, for which the irreducible cubic $V(\psi_r)$, from Claim 1, is not the projection of $V(\phi_r)$. The points of $V(\psi_1,\psi_2)$ are projected onto the projection of $V(\phi_r)$, so by Bezout's theorem, we conclude that there are at most $9$ points of the $n-O(K)$ points of Claim 1, that are in $V(\psi_1,\psi_2)$. Therefore, by Lemma~\ref{thisorthat}, there are at least $n-O(K)$ points of $S$ in $V(\phi_{pq})$. If $\phi_{pq}$ is degenerate at $p$, then the $n-O(K)$ points of $V(\phi_{pq}) \cap S$ are projected from $p$ onto a (possibly degenerate) conic, which contradicts Claim 1. Hence $\phi_{pq}$ is not degenerate at $p$. The quadric $V(\phi_{pq})$ contains the line joining $p$ and $q$ and since it is not degenerate at $p$ is contains at most one other line incident with $p$. To find another quadric containing at least $n-O(K)$ points of $S$, we repeat the entire process with a different $q$ (prohibiting $q$ when we apply the pigeon-hole argument), say $q'$, with the property that the line joining $p$ and $q'$ is not contained in $V(\phi_{pq})$. 

Therefore, we conclude that at least $n-O(K)$ points of $S$ are contained in the intersection of two linearly independent quadrics, which from now on will be $V(\psi_1,\psi_2)$, and Claim 2 is proven.

Claim 3: Every point of $S' \cap V(\psi_1,\psi_2)$ is incident with at least $\frac{3}{2}n-O(K)$ ordinary planes.

Suppose $p\in S'\cap V(\psi_1,\psi_2)$. The projection $S_p$ of $S$ from $p$ spans at most $dKn$ ordinary lines and Theorem~\ref{GTinter} applies. There are $(1-4/d)n-c'K$ points in $S'' \cap V(\psi_1,\psi_2)$ which are projected from $p$ onto the projection of $V(\phi_p)$, for some constant $c'$. If the projection of $V(\phi_p)$ is reducible then, by Theorem~\ref{GTinter}, there is a line of the projection which contains at least 
$$
((1-4/d)n-2^{87}(dK)^6-c'K)/2^{64}(dK)^4>c n/K^4
$$ 
points of the projection of $S'' \cap V(\psi_1,\psi_2)$, for some constant $c$ when $n$ is large enough. However, the projection of $S$ from a point of $S'' \cap V(\psi_1,\psi_2)$ is contained in an irreducible cubic and an additional $O(K)$ points by Claim 1. Since $K=o(n^{1/6})$, $c n/K^4>K^2$ for $n$ large enough, so there cannot be a line containing at least $cn/K^4$ points of the projection of $S'' \cap V(\psi_1,\psi_2)$. Hence, we deduce that $S_p$ is contained in an irreducible cubic $V(\psi_p)$, which by Bezout's theorem must be the projection of $V(\phi_p)$, and an additional $2^{75}(dK)^5$ points. This implies that $p \in S''$ and by Claim 1, the projection $S_p$ of $S$ from a point $p \in S'$ is contained in an irreducible cubic and an additional $O(K)$ points. Lemma 7.2 from \cite{GT2013} implies that there is a coset $x_p \oplus H_p$ of $V(\psi_p)$ such that the symmetric difference of $S_p$ and $x_p \oplus H_p$ has size $O(K)$. 

Let $(q)_p \in S_p \cap (x_p \oplus H_p)$. The tangent to $V(\psi_p)$ at $(q)_p$ is an ordinary line of $S_p$ for all but at most $O(K)$ points $q \in S'$. By Lemma~\ref{tangentproject}, the line $\ell_q$
is projected onto the tangent to $V(\psi_p)$ at $(q)_p$, since $V(\psi_p)$ is the projection of $V(\phi_p)$ from $p$. Hence, for all but at most $O(K)$ points $q$ of $S'$, the plane $\pi_p^q$, spanned by the line $\ell_q$ and the point $p$, is an ordinary plane of $S$.

Suppose $q \in S'$. Let $(\ell_q)_q$ be the projection of $\ell_q$ from $q$ which, by Lemma~\ref{tangentproject}, is a point of the projection of $V(\phi_q)$. Projecting from $q$, the plane $\pi_p^q$ is projected onto an ordinary line of $S_q$ for all but at most $O(K)$ points of $S'$. Moreover, these lines are all incident with the point $(\ell_q)_q$, since the planes $\pi_p^q$ all contain the line $\ell_q$. Therefore, we have that $S_q$ spans $\frac{1}{2}n-O(K)$ concurrent ordinary lines. Moreover, as we have already seen, the symmetric difference of $S_q$ and $x_q \oplus H_q$ has size $O(K)$ which implies that $S_q$ spans at least $n-O(K)$ other ordinary lines. Thus, $q$ is incident with at least $\frac{3}{2}n-O(K)$ ordinary planes.
\end{proof}

In the second lemma of its type, we consider the case in which $S$ has a few points which project the other points of $S$ onto a set of type (ii) of Theorem~\ref{GTinter}.

\begin{lemma} \label{typetwo} 
Let $S$ be a set of $n$ points in $\mathrm{PG}_3({\mathbb R})$, no three collinear and not all co-planar, spanning at most $Kn^2$ ordinary planes, where $K=o(n^{\frac{1}{7}})$. Let $S_{\sigma}$ be the subset of $S$ for which $p \in S_{\sigma}$ if and only if the projection of $S$ from $p$ spans at most $dKn$ ordinary lines and is contained in the union of an irreducible conic $\sigma_p$, containing at least $\frac{1}{2}n-2^{76}(dK)^5$ points of the projection, and a set $L_p$ of at most $2^{64}(dK)^4$ lines. If $|S_{\sigma}|>n/2d$, for some large constant $d$, then there are two planar conic sections of a quadric which each contain at least $\frac{1}{2}n-cK$ points of $S$, for some constant $c$.
\end{lemma}

\begin{proof}

This proof is long, so we break it down into five claims.

Let $p\in S_{\sigma}$.

Claim 1. There are two linearly independent quadrics $V(\psi_1)$ and $V(\psi_2)$ which contain $p$ and more than $2n/(15d^2K)$ points of $S$.

A point $q$ of $S$ projects from $p$ onto a point $(q)_p$. By hypothesis, $\sigma_p$ contains at least $\frac{1}{2}n-2^{76}(dK)^5$ points of $S_p$, the projection of $S$. Each point of $\sigma_p \cap S_p$ is joined to $(q)_p$ by a line which meets at most one of the other points of $S_p \cap \sigma_p$. Therefore, there are at least $\frac{1}{2}(\frac{1}{2}n-2^{76}(dK)^5)>\frac{1}{5}n$ lines incident with $(q)_p$ and a further point of the projection. Hence, there are at least $\frac{1}{5}n$ planes containing the line joining $p$ and $q$, which are incident with a further point of $S$. This implies that the line $p^* \cap q^*$ contains at least $\frac{1}{5}n$ edges of $\Gamma$. 

Let $S'$ be the points of $S$ which project $S$ onto a set spanning at most $dKn$ ordinary lines. By counting pairs $(x,\pi)$ where $x \in S \setminus S'$ and $\pi$ is an ordinary plane containing $x$,
$$
|S \setminus S'|dKn \leqslant 3Kn^2,
$$
so $|S'| \geqslant (1-3/d)n$. 

Let $S_p'$ be the set of points of $S'$ which are projected from $p$ onto $\sigma_p$ but are not incident with a line of $L_p$. There are at least
$$
\tfrac{1}{2}n-3n/d-2^{76}(dK)^5-2^{65}(dK)^4
$$
points in $S_p'$, since any line of $L_p$ is incident with at most two points of $\sigma_p$ and $|S \setminus S'| \leqslant 3n/d$.

Hence, for $n$ large enough, $|S_p'|>\frac{1}{3}n$. 

By Lemma~\ref{slightlybad}, there are at most $690Kn^2$ slightly bad edges in $\Gamma$. 

Let $S_{\sigma}'$ be the subset of $S_{\sigma}$ with the property that $p \in S_{\sigma}'$ if and only if $p^*$ contains at most $d^2Kn$ slightly bad edges. By counting pairs $(x,e)$ where $x \in S_{\sigma} \setminus S_{\sigma}'$ and $e$ is a slightly bad edge in $x^*$,
$$
|S_{\sigma} \setminus S_{\sigma}'|d^2Kn \leqslant 2\times 690 Kn^2.
$$
By choosing $d$ large enough, we can assume
$$
|S_{\sigma}'| \geqslant (\frac{1}{2d}-\frac{1380}{d^2})n>\frac{n}{d^2}.
$$ 

Let $p\in S_{\sigma}'$. Then $p^*$ contains at most $d^2Kn$ slightly bad edges and so
by the pigeon-hole principle, there is a $q \in S_p'$ such that $p^* \cap q^*$ contains at most $3d^2K$ slightly bad edges. Again, by the pigeon-hole principle, since the line $p^* \cap q^*$ contains at least $\frac{1}{5}n$ edges of $\Gamma$, there is a segment $T$ of $p^* \cap q^*$ of 
$$
m \geqslant \frac{n}{15d^2 K}
$$ 
rather good edges. By Lemma~\ref{segment}, and labelling the points in the same way as in the proof of Lemma~\ref{segment}, there are linearly independent quadrics $V(\psi_1)$ and $V(\psi_2)$ containing  
$$
M_p=\{ r_{-m},\ldots,r_{-1},r_0,s_m,\ldots,s_1,s_0,p_{-1},p_0,p_1,q_{-1},q_0,q_1 \},
$$
where $p=p_0$ and $q=q_0$ and $M_p$ is a subset of $S$. Claim 1 is proven.

As before, define
$$
\phi_p(x)=b_1(p,x)\psi_2(x)-b_2(p,x)\psi_1(x).
$$
Clearly $M_p$ is contained in $V(\phi_p)$.

Claim 2. There is a $p \in S_{\sigma}'$ for which the projection of $V(\phi_p)$ from $p$, is the union of a conic $\sigma_p$ and a line $\ell_p$.

If $\phi_p$ is irreducible then each line of $L_p$ meets the projection of $V(\phi_p)$ in at most three points and the conic $\sigma_p$ meets the projection of $V(\phi_p)$ in at most six points. However, the projection of $V(\phi_p)$ contains at least $2m+7$ points of the projection and
$$
3 \times 2^{64}K^4+6 < 2m+7,
$$
so we conclude that $\phi_p$ is reducible.

Suppose that for all $p \in S_{\sigma}'$, $\phi_p$ factorises into three linear forms. Then the projection of $V(\phi_p)$ is three lines. These three lines contain the projection of $M_p$, whose points have the property that the projection of $q_i, r_j, s_k$ are collinear if and only if $i+j+k=0$. Therefore, one of the lines contains the projection of $\{r_{-m},\ldots,r_{-1},r_0\}$ and one of the other lines contains the projection of $\{ s_m,\ldots,s_1,s_0\}$. So we have that there are two planes $\pi_p$ and $\pi_p'$, incident with $p$, which contain at least $n/15d^2K$ points of $S$. Let $R_{{\sigma},p}=S_{\sigma}' \cap \pi_p$. For each $q \in R_{{\sigma},p}$, at least one of the planes $\pi_q$ or $\pi_q'$ is not $\pi_p$, so assume that $\pi_q \neq \pi_p$. Then $\pi_q \cap \pi_p$ is a line which contains at most two points of $S$ and so at most two points of $R_{{\sigma},p}$. Hence, there are at least $\frac{1}{2}|R_{{\sigma},p}|$ planes containing at least $m$ points of $S$. By counting triples of points of $S$ we have
$$
{n \choose 3}=\sum_{i \geqslant 3} {i \choose 3} \tau_i,
$$
where $\tau_i$ is the number of planes incident with exactly $i$ points of $S$. Since $m \geqslant n/15d^2K$ we have that $|R_{{\sigma},p}|=O(K^3)$. Therefore, there are at least $n/16d^2K$ points of $S$ in $\pi_p \setminus S_{\sigma}$, for $n$ sufficiently large. We will now exploit this to prove that $S$ has more than $n$ points.

Let $T$ be a subset of $S_{\sigma}'$ of size $\lceil n/(d^3K) \rceil$. For any $q \in T  \setminus \pi_p$, the planes $\pi_q$ and $\pi_q'$ intersect $\pi_p$ in a line containing at most two points of $S$, so there are at least 
$$
\frac{n}{15d^2K}-\frac{4n}{d^3K}
$$
points of $S \cap \pi_p$, which are not incident with a plane $\pi_q$ or $\pi_{q}'$, where $q \in T  \setminus \pi_p$. For $n$ and $d$ large enough, this number is at least $n/20d^2K$.
These points may be be incident with a plane $\pi_q$ or $\pi_{q}'$, where $q \in T \cap \pi_p$. But, recall that we have just proved that $|T \cap \pi_p|=O(K^3)$, so they are incident with at most $O(K^3)$ planes $\pi_q$, where $q \in T$.

Counting points of $S$ on $\pi_p$, for all $p \in T$, we get
$$
cK^3|S|>\frac{n}{d^3K} \times \frac{n}{20d^2K},
$$ for some constant $c$, which is a contradiction since $K=o(n^{\frac{1}{7}})$ and $|S|=n$.

Therefore, there is a point $p \in S_{\sigma}$, for which the projection of $V(\phi_p)$ from $p$, is the union of a conic $\sigma_p$ and a line $\ell_p$ and Claim 2 is proven.

Claim 3. There are two plane conics $\sigma$ and $\sigma'$ each containing at least $m$ points of $S$.

By Lemma~\ref{doubled}, the structure around a rather good edge of $\Gamma$ is a tetra-grid, so we have that if $i,j \in \{-1,0,1\}$, $k \in \{-m,\ldots,0 \}$, $\ell \in \{0,\ldots,m \}$ and $i+j+k+\ell=0$, then the points $p_i, q_j, r_k,s_{\ell}$ are co-planar.

Since $q_0 \in S_p'$, $q_0$ projects from $p=p_0$ onto $\sigma_p$, and we have that one of $\{ (r_{-j})_p, (s_j)_p \}$ is on $\ell_p$ and the other on $\sigma_p$. Without loss of generality, assume that $(r_0)_p \in \ell_p$ and $(s_0)_p \in \sigma_p$. The line $\ell_p$ contains at least $m$ points of $\{r_{-j},s_j \ | \ j=1,\ldots m\}$, so the point $(q_{-1})_p \in \sigma_p$, since $q_{-1}$ is on planes containing the four points $p_0,q_{-1},r_{-j+1},s_j$ of $S$, for all $j=1,\ldots,m$. If $(s_1)_p \in \ell_p$ then, since $(r_0)_p \in \ell_p$ and $p_0,q_{-1},r_0,s_1$ are co-planar, this would imply that $(q_{-1})_p \in \ell_p$, which we have just shown it is not. Hence, $(s_1)_p \in \sigma_p$ and $(r_{-1})_p \in \ell_p$. Similarly $(s_j)_p \in \sigma_p$ and $(r_{-j})_p \in \ell_p$, for all $j=0,\ldots,m$.

Hence, there is a plane $\pi$ which contains $p_0,r_0\ldots,r_m$. These points lie in $V(\psi_1,\psi_2)$ too. Since $S$ does not contain three collinear points we conclude that $p_0,r_0\ldots,r_m$ are contained in an irreducible conic $\sigma$, contained in $V(\psi_1, \psi_2)$. Moreover, $q=q_0$ is not contained in $\sigma$.

Since $q=q_0 \in S'$, we can apply Theorem~\ref{GTinter} to the projection of $S$ from $q$. The projection of $S$ from $q$, projects the points of $\sigma$ onto a conic. There are at least $m$ points of $S$ on the conic $\sigma$, we can rule out case (i), since $m\geqslant n/(15d^2K)$. Likewise, we can rule out case (iii), since $K=o(n^{\frac{1}{7}})$ implies $m>2^{87}K^6+2^{17}K$ for $n$ sufficiently large. Therefore we are in case (ii), so $q \in S_{\sigma}$. The projection of $V(\phi_q)$ from $q$ cannot be three lines or an irreducible cubic since the $m$ points of $S \cap \sigma$ are projected onto a conic. Therefore, the projection of $\phi_q$ from $q$ is a conic and a line and we can repeat the first part of Claim 3, replacing $p$ with $q$ and conclude that there is an irreducible conic $\sigma'$, containing $q$ and more than $m$ other points of $S$. Hence, Claim 3 is proven. Let $\pi'$ denote the plane containing $\sigma'$.

Claim 4. The union $\sigma \cup \sigma'$ contains at least $n-2^{76}K^5$ points of $S$.

If there is a point $p_{\infty}$ that projects $\sigma$ onto $\sigma'$ then $\sigma$ and $\sigma'$ lie in the intersection of two linearly independent quadrics, one a plane pair and the other a conical cone. By Lemma~\ref{2pis}, there is at most one other point $p_{\infty}'$ which projects $\sigma$ onto $\sigma'$.

Suppose that $w$ is a point in $S' \setminus (\pi \cup \pi' \cup \{ p_{\infty}, p_{\infty}' \})$. Applying Theorem~\ref{GTinter}, with $K$ replaced by $dK$, we have that the projection of $S$ from $w$ is either one of (i), (ii) and (iii). The points of $S \cap \sigma$ and $S \cap \sigma'$ in $\pi \cup \pi'$ are projected onto two conics each containing at least $m$ points of the projection of $S$. As in the previous paragraph we can rule out cases (i) and (iii) and case (ii) is also ruled out since the additional $2^{64}K^4$ lines cannot contain $m$ points of a conic.

Therefore all points of $S' \setminus  \{ p_{\infty}, p_{\infty}' \} $ are contained in $\pi \cup \pi'$.

Now we apply Theorem~\ref{GTinter} to the the projection of $S$ from a point of $S' \cap \pi$. The $m$ points of $S \cap \sigma'$ are projected onto a conic so, arguing as before, we are in case (ii). Therefore, the points of $(S \cap\pi')\setminus \sigma'$ are contained in at most $2^{64}K^4$ lines. Since these points are co-planar, and $S$ does not contain three collinear points, we deduce that there are at most $2^{65}K^4$ of them. Likewise, projecting from a point of $S' \cap \pi'$ we see that at most $2^{65}K^4$ points of $S\cap \pi$ are not in the conic $\sigma$. Therefore, $\sigma \cup \sigma'$ contains all but at most $O(K^4)$ points of $S \cap (\pi \cup \pi')$.

Suppose that $p$ is a point of $S \setminus (S' \cup\{ p_{\infty},p_{\infty}'  \})$ which is not contained in $\pi \cup \pi'$. Let $N_p$ be the set of points of $\sigma'$ such that the projection of $\sigma$ from $p$ onto $\pi'$ meets $\sigma'$ in $N_p$. By Bezout's theorem, $N_p$ consists of at most four points. There are at most $3n/d+O(K^4)$ points in $S \setminus (\pi \cup \pi')$ so there is a point 
$q \in S' \cap \sigma'$ which is not in $N_p$ for any $p \in S \setminus (S' \cup \pi \cup \pi' \cup \{ p_{\infty},p_{\infty}'  \})$.

The projection of $S$ from $q$ onto $\pi$ projects $S \setminus (S' \cup \pi \cup \pi' \cup \{ p_{\infty} , p_{\infty}' \})$ onto the additional $K^4$ lines given by case (ii) of Theorem~\ref{GTinter}. But $q$ projects at least $\frac{1}{2}n - 2^{76}K^5$ points onto $\sigma$, according to Theorem~\ref{GTinter}, so there are at least $\frac{1}{2}n - 2^{76}K^5$ points in $\sigma \cap S$. Similarly, there are at least $\frac{1}{2}n - 2^{76}K^5$ points in $\sigma' \cap S$ and Claim 4 is proven.

Claim 5. The union $\sigma \cup \sigma'$ contains at least $n-O(K)$ points of $S$.

Let $X$ be a planar set of points which is, after a suitable projective transformation, the points of a regular polygon and the points at infinity where the bisecants to the regular polygon meet. 
By \cite[Lemma 7.4]{GT2013}, the projection of $S$ from a point of $\sigma \cap S'$, differs from $X$ by at most $O(K^5)$ points. By \cite[Lemma 7.6]{GT2013}, a point $u$ which is not in $X$, is incident with at least $|X|-O(1)$ lines which contain exactly one point of $X$. 

Let $W$ be the set of points of $S$ which are not contained in $\sigma \cup \sigma'$. Let $\epsilon=|W|$. Without loss of generality, $|W \setminus \pi|  \geqslant \frac{1}{2}\epsilon$. Let $w \in W \setminus (\pi \cup  \{ p_{\infty},p_{\infty}'  \})$. The projection of $\sigma$ from $w$ onto $\pi'$ meets $\sigma'$ in at most four points. So for $|S' \cap \sigma|-4$ points $q \in S' \cap \sigma$, $q$ projects $w$ onto a point of $\pi'$ incident with at least $n-cK^5-\epsilon$ ordinary lines, for some constant $c$. Therefore, $S$ spans at least 
$$
(\tfrac{1}{2}\epsilon -2) (|S'\cap \sigma|-4)(n-cK^5-\epsilon)
$$ 
ordinary planes. Since $|S' \cap \sigma|>\frac{1}{3}n$ and $S$ spans at most $Kn^2$ ordinary planes, we have that $\epsilon=O(K)$.

\end{proof}

In the third lemma of its type, we consider the case that $S$ has a few points which project the other points of $S$ onto a set of type (iii) of Theorem~\ref{GTinter}.

\begin{lemma} \label{typethree} 
Let $S$ be a set of $n$ points in $\mathrm{PG}_3({\mathbb R})$, no three collinear and not all co-planar, spanning at most $Kn^2$ ordinary planes, where $K=o(n^{\frac{1}{6}})$. Suppose that there are three points of $S$ which project $S$ onto a set spanning at most $dKn$ ordinary lines, for some constant $d$, and which is contained in the union of $2^{16}K$ lines and an additional $2^{87}K^6$ points. Then all but at most $2K$ points of $S$ are contained in a plane.
\end{lemma}

\begin{proof}
Let $p,q,r$ be the three points of $S$ which project $S$ onto a set spanning at most $dKn$ ordinary lines, for some constant $d$, and which is contained in the union of $2^{16}K$ lines and an additional $2^{87}K^6$ points. Let $L_p$ denote the set of planes, incident with $p$, and projected from $p$ onto one of the $2^{16}K$ lines. Similarly, define $L_q$ and $L_r$.

There are at most $2(2^{16}K)^2$ points of $S$ in the intersection of a plane of $L_p$ containing $p$ and not $q$, with a plane of $L_q$ containing $q$ and not $p$, since $S$ contains no three collinear points. Therefore, there are at least $n-2^{87}K^6-2^{33}K^2$ points of $S$ contained in at most $2^{16}K$ planes containing the line $\ell_{pq}$, which joins $p$ and $q$.

Similarly, there are at least $n-2^{87}K^6-2^{33}K^2$ points of $S$ contained in at most $2^{16}K$ planes containing the line $\ell_{pr}$, the line joining $p$ and $r$. However, there are at most $2(2^{16}K)^2$ points in the intersection of a plane of $L_r$ containing $r$ and not $q$, with a plane of $L_q$ containing $q$ and not $r$, so we conclude that at least $n-2^{87}K^6-2^{34}K^2$ points of $S$ are contained in the plane $\pi$ spanned by $p$, $q$ and $r$.

Let $\epsilon$ be the number of points in $S \setminus \pi$. By the previous paragraph, $\epsilon=o(n)$. Each point of $S\setminus \pi$ is incident with at least ${n-\epsilon \choose 2}-\epsilon$ ordinary planes which contain exactly one point of $S \setminus \pi$. Hence, there are at least $\epsilon({n-\epsilon \choose 2}-\epsilon)$ ordinary planes. Since $S$ has at most $Kn^2$ ordinary planes, we have that $\epsilon \leqslant 2K$, for $n$ large enough. 
\end{proof}

\begin{proof} (of Theorem~\ref{inter})

Let $d$ be a large constant and let $S'$ be the points of $S$ that project $S$ onto a set with at most $dKn$ ordinary lines. By counting pairs $(x,\pi)$ where $x \in S \setminus S'$ and $\pi$ is an ordinary plane containing $x$,
$$
|S \setminus S'|dKn \leqslant 3Kn^2,
$$
so $|S'| \geqslant (1-3/d)n$. 

If there are $(1-4/d)n$ points of $S'$ of type (i) of Theorem~\ref{GTinter} then Lemma~\ref{typetwo} applies and we have conclusion (i).

If there are $n/(2d)$ points of $S'$ of type (ii) of Theorem~\ref{GTinter} then Lemma~\ref{typetwo} applies and we have conclusion (ii). 

If there are three points of $S'$ of type (iii) of Theorem~\ref{GTinter} then Lemma~\ref{typethree} applies and we have conclusion (iii). 
\end{proof}

\section{The detailed structure theorem}

In this section we shall prove Theorem~\ref{detailed} by analysing case (ii) of Theorem~\ref{inter}.

\begin{proof} (of Theorem~\ref{detailed}) By hypothesis, $S$ spans at most $\frac{1}{2}n^2-cn$ ordinary planes. 

If we are in case (i) of Theorem~\ref{inter} then $S$ would span at least
$$
\tfrac{1}{3}(\tfrac{3}{2}n-c')(n-c'')
$$
ordinary planes, for some constants $c', c''$, which is a contradiction for $c$ large enough.

If we are in case (iii) of Theorem~\ref{inter} then $S$ would span at least
$$
\tfrac{1}{2}n^2-dn
$$
ordinary planes, for some constant $d$, which is a contradiction for $c>d$.

Therefore, we can assume that we are in case (ii) of Theorem~\ref{inter}. Theorem~\ref{inter} implies that there is a quadric $\psi$ whose conic sections $\psi \cap \pi_1$ and $\psi \cap \pi_2$ contain all but $O(1)$ points of $S$.

We can choose a basis so that $\pi_1$ is the plane $X_4=0$ and $\pi_2$ is the plane $X_3=0$. After a suitable projective transformation we can assume that the conic section in $\pi_1$ is the set of solutions of the equation 
$$
X_1^2+X_2^2=X_3^2.
$$
The quadric $\psi$ is the set of zeros of a quadratic form
$$
X_1^2+X_2^2-X_3^2-eX_4^2+2aX_1X_4+2bX_2X_4+dX_3X_4,
$$
for some $a,b,e,d \in {\mathbb R}$.
 
Since $\psi$ intersects the plane $\pi_2$ in a conic, the equation
$$
(X_1+aX_4)^2+(X_2+bX_4)^2=(e+a^2+b^2)X_4^2
$$
must have solutions. Therefore, $e>0$ and we can scale the basis elements so that $e+a^2+b^2=1$.

Let $S_i=S\cap \pi_i \cap \psi$, for $i=1,2$. By Theorem~\ref{inter}, we have that $|S_i| \geqslant \frac{1}{2}n-c$, for some constant $c$. Let $p$ be a point of $S_2$ incident with at most $dn$ ordinary planes, for some constant $d$. By \cite[Lemma 7.4]{GT2013}, the projection of $S$ from $p$ differs from a B\"or\"oczky example by at most $O(1)$ points. Therefore, $S_1$ contains at least $\frac{1}{2}n-O(1)$ points of 
a regular $m_1$-gon, for some $m_1 \geqslant \frac{1}{2}n-c'$, for some constant $c'$.
 
As in Section~\ref{prism}, we define
$$
p_i=(\cos (2\pi i/m_1), \sin (2 \pi i/m_1), 0 , 1 )
$$ 
and define
$$
R_1=\{ p_i \ | \ i=0,\ldots,m_1-1 \}.
$$
Then $S_1$ differs from $R_1$ by at most $O(1)$ points.

Let
$$
q_k=(\cos (2\pi (k+\delta)/m_2)-a, \sin (2 \pi (k+\delta)/m_2)-b, 0 , 1 )
$$
for some $0 \leqslant \delta <1$ and define
$$
R_{2}=\{ q_k  \ | \ k=0,\ldots,m_2-1 \}.
$$
Then $S_2$ differs from $R_2$ by at most $O(1)$ points.

The line joining the points $p_i$ and $p_j$ meets the line $\ell_{\infty}= \pi_1 \cap \pi_2$ in the point
$$
s_{\theta}=(\sin (\pi \theta/m_1), \cos (\pi \theta/m_1),0,0),
$$
where $\theta=-(i+j)$ mod $m_1$. 

Let $D_1=\{ s_{\theta} \ | \ \theta=0,\ldots,m_1-1 \}$.

The line joining the points $q_k$ and $q_{\ell}$ meets the line $\ell_{\infty}= \pi_1 \cap \pi_2$ in the point
$$
t_{\theta}=(\sin (\pi \theta/m_2), \cos (\pi \theta/m_2),0,0),
$$
where $\theta=-(k+\ell+2\delta)$ mod $m_2$.

Let $D_2=\{ t_{\theta} \ | \ \theta=0,\ldots,m_2-1 \}$.

Any point of $D_1 \setminus D_2$ or $D_2 \setminus D_1$ is incident with at least $\frac{1}{4}n^2-O(n)$ ordinary planes. Since $S$ spans at most $\frac{1}{2}n^2-cn$ ordinary planes, there are at most three such points. Since both $m_1$ and $m_2$ are at least $\frac{1}{2}n-c'$, for some constant $c'$, we conclude that $D_1=D_2$ and so $m_1=m_2=m$.

If the points $p_i$, $p_j$, $q_k$ span a non-ordinary plane then there is a point $q_{\ell}$ incident with the plane, for some $\ell \neq k$, where
$$
i+j=k+\ell+2\delta.
$$
Since $i$, $j$, $k$ and $\ell$ are all integers, we have that either $\delta=0$ and $S$ differs by $O(1)$ points from a prism, or $\delta=\frac{1}{2}$ and $S$ differs by $O(1)$ points from an anti-prism.

Let $X$ denote the prism or anti-prism containing $n-O(1)$ points of $S$ and suppose that $x \in S \setminus X$. We will prove that $x$ is incident with at least $\frac{1}{4}n^2-O(n)$ ordinary planes. Assuming this to be the case, since there are at least $\frac{1}{4}n^2-O(n)$ ordinary planes spanned by $X \cap S$, this implies that $S$ is contained in $X$. Furthermore, each point of $X \setminus S$ is incident with $\frac{1}{8}n^2-O(n)$ ordinary planes, since each point of $X$ is incident with $\frac{1}{8}n^2-O(n)$ planes containing four points of $X$. Therefore, we conclude that $|X\setminus S| \leqslant 1$.

It remains to prove that $x\in S \setminus X$ is incident with at least $\frac{1}{4}n^2-O(n)$ ordinary planes.

By Lemma~\ref{2pis}, there are at most two points that project $\sigma$ onto $\sigma'$, which we denote by $p_{\infty}$ and $p_{\infty}'$.

Suppose first that $x \not\in \pi_1 \cup \pi_2 \cup \{p_{\infty}, p_{\infty}' \}$. Let $p$ be a point of $S \cap X \cap \pi_1$ and let $x_p$ denote the point of the projection of $x$ from $p$ onto the plane $\pi_2$. If $x_p \not\in X$ or is not the centre of the regular polygon $X \cap \pi_2$ then, by \cite[Corollary 7.6]{GT2013}, $x_p$ is incident with at least $\frac{1}{2}n+c$ lines incident with just one point of $X \cap \pi_2$, for some constant $c$. There are at most four points in $X$ that project $x$ onto $\sigma'$, since $x$ projects $\sigma$ onto a conic meeting $\sigma'$ in at most four points. Therefore, $x$ is incident with at least $(\frac{1}{2}n-c')(\frac{1}{2}n+c)=\frac{1}{4}n^2-O(n)$ ordinary planes, for some constant $c'$.

Suppose that $x\in \pi_1 \cup \pi_2$ and without loss of generality, that $x\in \pi_1$. Let $Y$ be the set of points on $\pi_1 \cap \pi_2$ incident with planes that are incident with four points of $X$. (So $(X \cap \pi_1) \cup Y$ is a B\"or\"oczky example in the plane $\pi_1$). By \cite[Corollary 7.6]{GT2013}, there are $n-O(1)$ lines incident with $x$ and precisely one point of $(X \cap \pi_1) \cup Y$, and so there are $\frac{1}{2}n-O(1)$ lines $\ell$ of $\pi_1$ incident with $x$ and precisely one point of $X \cap \pi_1$ which intersect $\pi_1 \cap \pi_2$ in a point not in $Y$. Therefore, for any point $p \in \pi_2 \cap S$, the plane spanned by $\ell$ and $p$ is an ordinary plane incident with $x$. Thus, $x$ is incident with at least $\frac{1}{4}n^2-O(n)$ ordinary planes.

Finally, suppose that $x \in \{p_{\infty}, p_{\infty}' \}$. Then $X$ must be an anti-prism, since $S$ does not contain three collinear points. But then every plane spanned by $x$, a point of $S \cap X \cap \pi_1$ and a point of $S \cap X \cap \pi_2$ is ordinary.

\end{proof}

\section{Ordinary circles} \label{ordcircles}

Let $S$ be a set of points in the real plane. For any subset of two points of $S$, recall that we say that the unique line passing though them is {\em ordinary} if it contains no other points of $S$. For any subset of three non-collinear points of $S$, we say that the unique circle passing though them is {\em ordinary} if it contains no other points of $S$. 

Theorem~\ref{inter} has the following corollary. A more detailed structure theorem for ordinary circles, with a stronger hypothesis, is obtained in \cite{LMMSSdZ2016}.

\begin{theorem}
Let $S$ be a set of $n$ points in the real plane spanning at most $Kn^2$ ordinary circles, for some function $K=K(n)$. If $K=o(n^{\frac{1}{7}})$ then, for $n$ large enough, all but at most $O(K)$ points of $S$ are contained in a curve of degree at most four.
\end{theorem}

\begin{proof}
Embed the plane in three-dimensional space and consider a sphere $\Omega$ whose tangent plane $\pi$ at the south pole contains the set $S$. Let $T$ be the set of points obtained by inversely projecting $S$ from the north pole $x$ onto $\Omega$, together with the point $x$. By the well-known property of stereographic projection, circles in $\pi$ are lifted to plane sections of $\Omega$. An ordinary circle spanned by $S$ lifts to an ordinary plane spanned by $T$ not incident with $x$, and conversely an ordinary plane spanned by $T$ not incident with $x$ is projected onto an ordinary circle spanned by $S$.  An ordinary line spanned by $S$ lifts to an ordinary plane spanned by $T$ incident with $x$, and conversely an ordinary plane spanned by $T$ incident with $x$ is projected onto an ordinary line spanned by $S$.  
Since $S$ spans at most $\frac{1}{2}n^2+O(n)$ ordinary lines we conclude that $T$ spans at most $(K+\frac{1}{2})n^2+O(n)$ ordinary planes. 

Since $T$ is contained in the sphere $\Omega$, it does not contain three collinear points and the points are not co-planar. Therefore, by Theorem~\ref{inter}, all but at most $O(K)$ points of $T$ are contained in the intersection of two quadrics. Let $\Omega'$ denote any other quadric containing all but at most $O(K)$ points of $T$. 

Let $\psi$ and $\psi'$ denote the quadratic forms which define the quadrics $\Omega$ and $\Omega'$ respectively.

After a suitable choice of basis, we can assume that $x$ is the projective point $(0,0,0,1)$. Then 
$$
\psi=\phi(x_1,x_2,x_3)+x_4\alpha(x_1,x_2,x_3)
$$
for some quadratic form $\phi$ in ${\mathbb R}[x_1,x_2,x_3]$ and linear form $\alpha$ in ${\mathbb R}[x_1,x_2,x_3]$, and
$$
\psi'=\phi'(x_1,x_2,x_3)+x_4\alpha'(x_1,x_2,x_3)+\gamma x_4^2,
$$
for some quadratic form $\phi'$ in ${\mathbb R}[x_1,x_2,x_3]$, linear form $\alpha'$ in ${\mathbb R}[x_1,x_2,x_3]$, and $\gamma \in {\mathbb R}$.

Computing modulo $\psi$,
$$
\alpha^2 \psi'=\alpha^2\phi'-\alpha\alpha' \phi+\phi^2 \gamma.
$$
If this polynomial is zero then $\alpha^2 \psi'= \phi''\psi$, for some quadratic form $\phi''$. Since the set of zeros of $\psi$ is a sphere, $\alpha$ does not divide $\psi$. Therefore, $\phi''$ is a multiple of $\alpha^2$, which implies $\psi'$ is a multiple of $\psi$, which it is not since they are linearly independent.
 
The set of zeros of this non-zero homogeneous polynomial in $x_1,x_2,x_3$ of degree four contains all but $O(K)$ of the points of $S$. Note that, it may be the case that the polynomial factorises and the points of $S$ are contained in the set of zeros of a lower degree polynomial.
\end{proof}

Conversely, there are examples of sets $S$ of $n$ points contained in curves of degree $2$, $3$ and $4$, which span less than $\frac{1}{2}n^2+O(n)$ ordinary circles, see \cite{LMMSSdZ2016}. For example, in \cite[Section 4.1]{LMMSSdZ2016} there is a set of $n$ points in the plane contained in an ellipse, which spans at most $\frac{1}{2}n^2+O(n)$ ordinary circles. Lifting this example to the sphere $\Omega$ by inversely projecting from the north pole $x$, as in the proof above, and then projecting stereographically from a point of $T \setminus \{x\}$, one obtains a planar set which spans $\frac{1}{2}n^2+O(n)$ ordinary circles, contained in a curve of degree three. Projecting from a point of $\Omega$ which does not lie on the other quadric containing $T$ gives a planar set which spans $\frac{1}{2}n^2+O(n)$ ordinary circles, contained in a curve of degree four. In \cite{LMMSSdZ2016}, the precise minimum number of ordinary circles is determined and all point sets which come close to this minimum are classified.

\section*{Appendix} \label{chasles}

In this appendix we include a proof Theorem~\ref{7gives8}.

Let $T$ be a subset of points of $\mathrm{PG}_{k-1}({\mathbb R})$. Define $I(T)$ to be the subspace of homogeneous polynomials of degree $t$ in $k$ variables with the property that $f \in I(T)$ if and only if $T \subseteq V(f)$.

\begin{theorem} \label{chaslesextend}
Let $S$ be a set of ${k+t-1 \choose t}-r+1$ points in $\mathrm{PG}_{k-1}({\mathbb R})$, for some positive integer $r$. Let $S'$ be a subset of $S$ of size $|S|-1$.
Suppose that the dimension of $I(S')$ is $r$ and 
$$
S \subseteq V(f_1,\ldots,f_r),
$$
for some $f_1,\ldots,f_r$, linearly independent homogeneous polynomials in $k$ variables of degree $t$.

If $S' \subseteq V(f)$, for some homogeneous polynomial $f$ of degree $t$ then $S \subseteq V(f)$.
\end{theorem}

\begin{proof}
Since $f_1,\ldots,f_r$ are linearly independent homogeneous polynomials of degree $t$, they span an $r$-dimensional subspace $U$ of the space of homogeneous polynomials of degree $t$ in $k$ variables. Since $U \subseteq I(S')$ and the dimension of $I(S')$ is $r$, we have that $U=I(S')$. Since $S' \subseteq V(f)$, we have that $f \in I(S')=U$. Hence, $f$ is a linear combination of  $f_1,\ldots,f_r$ and since $S \subseteq V(f_1,\ldots,f_r)$, we have that $S \subseteq V(f)$.
\end{proof}

We will now use Theorem~\ref{chaslesextend} to prove Theorem~\ref{7gives8}.

\begin{proof} (of Theorem~\ref{7gives8})

Let $S$ be the eight points of $\mathrm{PG}_3({\mathbb R})$,
$$
S=\{ \pi_i \cap \pi_j' \cap \pi''_k \ | \ i,j,k \in \{1,2 \}\}.
$$ 

To apply Theorem~\ref{chaslesextend} with $k=4$, $t=2$ and $r=3$, we have to show that for any subset $S'$ of $S$ of size $7$, the seven points of $S'$ impose $7$ linearly independent conditions on the space of quadrics. This then implies that $I(S')$ has dimension 3. By Theorem~\ref{chaslesextend}, we have that for any quadric passing through the seven points of $S'$ passes through the eight points of $S$. 

It remains to show that for any subset $S'$ of $S$ of size $7$, the seven points of $S'$ impose $7$ linearly independent conditions on the space of quadrics. 

We consider two cases, either two of the lines $\ell_1=\pi_1 \cap \pi_2$, $\ell_2=\pi_1' \cap \pi_2'$ and $\ell_3=\pi_1'' \cap \pi_2''$ are skew or they pairwise intersect.

In the case that two of the lines, $\ell_1$ and $\ell_2$ say, are skew we can fix a basis so that $\{ \pi_1,\pi_2 \} $, $\{ \pi_1',\pi_2' \} $  and $\{ \pi_1'',\pi_2'' \}$ are the kernels of the linear forms $X_1,X_2$ and $X_3,X_4$ and $\alpha_1(X),\alpha_2(X)$ respectively. 

The hypothesis that the intersection of the planes gives eight distinct points implies that, with respect to the basis, the points of $S$ are
$$
\{  (0,1,0,a_{13}) ,  (0,1,0,b_{13}) , (1,0,0,a_{23}) ,  (1,0,0,b_{23}) ,$$
$$
 (0,1,a_{14},0) ,  (0,1,b_{14},0) ,  (1,0,a_{24},0),  (1,0,b_{24},0)  \}.
$$
and $0 \neq a_{13} \neq b_{13} \neq 0$, etc.

Let $\psi$ be a homogeneous polynomial of degree two in four variables,
$$
\psi(X)=c_1X_1^2+\cdots+c_4X_4^2+c_{12}X_{1}X_2+\cdots+c_{34}X_3X_4.
$$
If we suppose $\psi \in I(S')$ then we get seven equations imposing conditions on the coefficients of $\psi$,  by considering $\psi(x)=0$, for all $x \in S'$. Without loss of generality, suppose that $S \setminus S'=\{  (1,0,b_{24},0)  \}$. The seven equations, written in matrix form give the matrix equation
$$
\left( \begin{array}{cccccccccc}
1 & 0 & a_{24}^2 & 0 & a_{24} & 0 & 0 & 0 & 0 & 0\\
 1 & 0 & 0 & a_{23}^2 & 0 & 0 & a_{23} & 0 & 0 & 0 \\
  1 & 0 & 0 & b_{23}^2 & 0 & 0 & b_{23} & 0 & 0 & 0 \\
   0 & 1 & a_{14}^2 & 0 & 0 & a_{14} & 0 & 0 & 0 & 0  \\
      0 & 1 & b_{14}^2 & 0 & 0 & b_{14} & 0 & 0 & 0 & 0  \\
      0 & 1 & 0 & a_{13}^2 & 0 & 0 & 0 &  a_{13} & 0 & 0 \\
             0 & 1 & 0 & b_{13}^2 & 0 & 0 & 0 &  b_{13} & 0 & 0 \\
\end{array}\right) \left(\begin{array}{c} c_1 \\ c_2 \\ c_3 \\ c_4 \\ c_{13} \\ c_{23} \\ c_{14} \\ c_{24} \\ c_{12} \\ c_{34} \end{array} \right)=0.
$$
Since this matrix has rank $7$, we conclude that $\dim I(S')=3$.

In the case that $\ell_1$, $\ell_2$ and $\ell_3$ pairwise meet, we can choose a basis so that the three intersection points $\ell_1 \cap \ell_2$, $\ell_1 \cap \ell_3$ and $\ell_2 \cap \ell_3$ are $  (0,0,1,0)  ,   (0,1,0,0)  ,   (1,0,0,0)  $ and $\pi_1 \cap \pi_1' \cap \pi_1''=   (0,0,0,1)  $. This implies that $\pi_1=\ker X_1$ and $\pi_1'=\ker(X_1+a_1X_4)$, for some $a_1 \in {\mathbb R}$, etc. 

With respect to the basis, the points of $S$ are
$$
\{   (0,0,0,1)  ,   (0,0,-a_3,1)  ,   (0,-a_2,0,1)  ,    (-a_1,0,0,1)  ,$$
$$
    (0,-a_2,-a_3,1)  ,   (-a_1,0,-a_3,1)  ,   (-a_1,-a_2,0,1)  ,    (-a_1,-a_2,-a_3,1)   \}.
$$
If we suppose $\psi \in I(S)$ then we get eight equations imposing conditions on the coefficients of $\psi$,  by considering $\psi(x)=0$, for all $x \in S$. The eight equations, written in matrix form give the matrix equation
$$
\left( \begin{array}{cccccccccc}
0 & 0 & 0 & 1 & 0 & 0 & 0 & 0 & 0 & 0 \\
a_1^2 & 0 & 0 & 1 & 0 & 0 & 0 & -a_1 & 0 & 0  \\
 0 & a_2^2 & 0 & 1 & 0 & 0 & 0 & 0 & -a_2 & 0 \\
0 & 0 & a_3^2 & 1 & 0 & 0 & 0 & 0 & 0 & -a_3  \\
a_1^2 & a_2^2 & 0 & 1 & a_1a_2 & 0 & 0 & -a_1 & -a_2 & 0 \\
a_1^2 & 0 & a_3^2 & 1 & 0 & a_1a_3 & 0 & -a_1 & 0 & -a_3  \\
0 & a_2^2  & a_3^2 & 1 & 0 & 0 & a_2a_3  & 0 & -a_2  & -a_3  \\
a_1^2 & a_2^2  & a_3^2 & 1 & a_1a_2 & a_1a_3 & a_2a_3  & -a_1 & -a_2  & -a_3  \\
\end{array}\right) \left(\begin{array}{c} c_1 \\ c_2 \\ c_3 \\ c_4 \\ c_{12} \\ c_{13} \\ c_{23} \\ c_{14} \\ c_{24} \\ c_{34} \end{array} \right)=0.
$$
The kernel of the row space has dimension one and is spanned by
$$
(-1,1,1,1,-1,-1,-1,1).
$$
Hence, removing any row from this matrix gives a matrix of rank $7$, so we conclude that $\dim I(S')=3$.
\end{proof}

\section{Acknowledgments}

I am grateful to Konrad Swanepoel and Aaron Lin for informing me that there are examples of sets $S$ of $n$ points which span $\frac{1}{2}n^2+O(n)$ ordinary planes which are not prisms, anti-prisms and have at most four co-planar points. This implies that the bound in Theorem~\ref{detailed} is optimal. Their example is contained in the intersection of an elliptic cone and a sphere and is obtained from \cite[Example 4.1]{LMMSSdZ2016} by inverse stereographic projection, as explained in Section~\ref{ordcircles}. This construction led me to discovering a mistake in an earlier version of this article, where I had assumed that the projection of $\ell_p$ from $p$ was not a point of the projection of $\phi_p$. I would like to thank Frank de Zeeuw who also discovered this error. 

I would like to thank Tim Penttila for pointing out the reference \cite{Pedoe1988} for Theorem~\ref{7gives8} and that the ``eight associated points theorem'' can be found in Ottone Hesse's 1840 publication \cite{Hesse1840} and even further back in work by Gabriel Lam\'e from 1818. 

Finally, I would like to thank the anonymous referees for their comments, suggestions and queries. They led to significant improvements and I am very grateful that the referees took the time to read such a long article.

\bigskip

{\small Simeon Ball}  \\
{\small Departament de Matem\`atiques}, \\
{\small Universitat Polit\`ecnica de Catalunya, Jordi Girona 1-3},
{\small M\`odul C3, Campus Nord,}\\
{\small 08034 Barcelona, Spain} \\
{\small {\tt simeon@ma4.upc.edu}}

\end{document}